\documentclass{elsart}

\usepackage{color}
\usepackage{multirow}
\usepackage{booktabs}
\usepackage{graphicx}
\usepackage{hhline}
\usepackage{amsmath}
\usepackage{amsfonts}
\usepackage{amssymb}
\usepackage{enumitem}
\usepackage{subcaption}
\usepackage{algpseudocode}
\usepackage{dsfont}
\usepackage{bbm}
\usepackage{url}
\makeatletter
\def\amsbb{\use@mathgroup \M@U \symAMSb}
\makeatother
\usepackage{bbold}
\makeatletter
\def\BState{\State\hskip-\ALG@thistlm}
\makeatother
\newtheorem{theorem}{Theorem}

\newtheorem{lemma}[theorem]{Lemma}

\newenvironment{proof}{\noindent{\it Proof.}}{\hfill$\square$}
\begin{document}
\begin{frontmatter}

\title{A highly parallel algorithm for computing the action of a matrix exponential on a vector based
on a multilevel Monte Carlo method}

\author[add1,add2]{Juan~A.~Acebr\'on},
\ead{juan.acebron@iscte-iul.pt}
\author[add3]{Jos\'e R. Herrero},
\ead{josepr@ac.upc.edu}
\and
\author[add2]{Jos\'e  Monteiro}
\ead{jcm@inesc-id.pt}

\address[add1]{Dept. Information Science and Technology, ISCTE-University Institute of Lisbon, Portugal} 
\address[add2]{INESC-ID,Instituto Superior T\'ecnico, Universidade de Lisboa, Portugal} 
\address[add3]{Dept. d$'$Arquitectura de Computadors, Universitat Polit\`ecnica de Catalunya, Spain}

\begin{abstract}
A novel algorithm for computing the action of a matrix exponential over a vector is proposed. The algorithm is based
on a multilevel Monte Carlo method, and the vector solution is computed probabilistically 
generating suitable random paths which evolve through the indices of the matrix according to a suitable probability law.
The computational complexity is proved in this paper to be significantly better than the classical
Monte Carlo method, which allows the computation of much more accurate solutions. Furthermore, 
the positive features of the algorithm in terms of parallelism were exploited in practice to develop a highly 
scalable implementation capable
of solving some test problems very efficiently using high performance supercomputers equipped with a large number of cores.
For the specific case of shared memory architectures the performance of the algorithm was compared with the results
obtained using an available Krylov-based algorithm, outperforming the latter in all benchmarks analyzed so far.

\end{abstract}

\end{frontmatter}

\section{Introduction}

In contrast to the numerical methods for solving linear algebra problems, the development of methods for evaluating function of matrices has been in general 
much less explored. This can be explained partially due to the underlying mathematical complexity of evaluating the function, but also under the computational point of view, because
the algorithms developed so far tend to be less efficient and in general more difficult to be parallelized. 
In addition, related to the first issue, an added difficulty appears in estimating the associated error of the numerical method, which is well understood for solving iteratively linear algebra problems, but becomes a rather cumbersome process for functions of matrices. This is even worse in the case of the matrix exponential, due to the lack of a clear and consensually agreed notion of the residual of the iterative method, see for instance \cite{Botchev}.

In particular the second issue represents indeed a serious drawback, since it is preventing in practice to deal with large scale problems appearing in science and engineering. Nowadays there are a plethora of applications described by mathematical models which require evaluating some type of function of matrices in order to be solved numerically. 
For the specific case of the matrix exponential, we can find applications in fields as 
diverse as circuit simulations~\cite{Cheng1}; power grid simulations \cite{Sadiku,Cheng2}; nuclear reaction simulations 
\cite{Pusa}; analysis of transient solutions in Markov chains \cite{Sidjea}; numerical solution of partial differential equations (PDEs) \cite{Mattheij}; and analysis of complex networks~\cite{Benzi}, to cite just a few examples. More specifically, in the field of partial differential equations, numerically solving a boundary-value PDE problem by means of the method of lines requires in practice to compute the action of a matrix exponential using therefore exponential integrators \cite{Martinez}. On the other hand, in network analysis, determining some relevant metrics of the network, such as for instance the total communicability which characterizes the importance of the nodes inside 
the network, entails computing the exponential of the adjacency matrix of the network.

For the specific problem of computing the action of the matrix exponential over a vector several classes of numerical methods have been proposed in the literature in the last decades (see the excellent review in \cite{Higham}, and references therein). Probably the most analyzed and disseminated 
methods are those based on Krylov-based subspace methods, which use in practice a basis of a subspace constructed using the Arnoldi process, and compute the exponential of the projected matrix (typically much smaller) by using standard matrix exponential techniques \cite{Higham2}.

An alternative to the aforementioned deterministic methods does exist, and consists in using probabilistic methods based on Monte Carlo (MC) simulations.
Although much less known than the former methods, the Monte Carlo methods specifically used for solving linear algebra problems have been 
discussed in the literature in various forms along the years. In fact, it was the seminal paper by von Neumann and Ulam during the 40's \cite{Forsythe} that gives rise to an entire 
new field, and from there a multitude of relevant results, and substantial improvements of the original algorithm have appeared in the literature during the last years, see e.g. \cite{dimov2}  and \cite{dimov} for further references. Essentially the main goal is to generate a discrete Markov chain 
whose underlying random paths evolve through the different indices of the matrix. The method can be understood formally as a procedure consisting in a Monte Carlo sampling of the 
Neumann series of the inverse of the matrix. The convergence of the method was rigorously established in \cite{Mascagni}, and improved further more recently (see for instance \cite{Benzi3}, and \cite{dimov4} just to cite a few references).

Generalizing the method for dealing with some functions of matrices, such as the matrix exponential, was only recently accomplished in \cite{Acebron_SISC}. The method is based on generating random paths, which evolve through the indices of the matrix, governed now by a suitable continuous-time Markov chain. The vector solution is computed probabilistically by averaging over a suitable multiplicative functional.

The main advantages of the probabilistic methods, as it was already stated in the literature, are mainly due to its
privileged computational features, such as simplicity to code and parallelize. This in practice allows us to develop parallel codes with extremely low communication overhead among processors, having a positive impact in parallel features such as scalability and fault-tolerance.
Furthermore, there is also another distinguishing aspect of the method, which is the capability of computing the solution of the problem at specific chosen points, without the need for solving globally the entire problem. This remarkable feature has been explored for efficiently solving continuous problems such as boundary-value problems for PDEs in \cite{Acebron1,Acebron3,Acebron4}, offering significant advantages in dealing with some specific applications found in science and engineering.

Yet an important disadvantage of any Monte Carlo method is the slow convergence rate to the solution of the numerical method \cite{Evans}, being
in general of order $\mathcal{O}(N^{-1/2})$, where $N$ denotes the sample size. Nevertheless, there already exist a few statistical techniques, such as variance reduction, multilevel Monte Carlo (MLMC), and 
quasi-random numbers, which have been proposed to mitigate in practice such a poor performance, improving the order of the global error, and consequently the overall performance of the algorithm. 
Among all the aforementioned methods, the multilevel method clearly stands out, and currently it has become in fact the preferred method to speed up the convergence of a variety of stochastic simulations, with a remarkable impact on a wide spectrum of applications. 
An excellent review has been recently published in \cite{Giles1} describing in detail the method as well as a variety of applications where it was successfully applied
(see also \cite{Anderson} for more details specifically related with the topic of this paper).

One of the main contributions of this paper is precisely to develop a multilevel method for the problem of computing the action of a matrix exponential over a vector. This is done by conveniently adapting the probabilistic representation of the solution derived in \cite{Acebron_SISC} to the multilevel framework. In addition, the convergence of the method is analyzed, as well as the computational cost
estimated. The second important contribution was to parallelize the resulting algorithm, and finally run successfully several relevant benchmarks for an extremely large number of processors using high performance supercomputers belonging to the top-performance
supercomputers in the world (according to the well-known {\it TOP500} list \cite{top500}).

The outline of the paper is as follows. Briefly, the mathematical description of the probabilistic method   
is summarized in Section \ref{theory}, and the problem is mathematically formalized according to the multilevel framework. 
In Section \ref{algorithm}, the developed algorithm is described through the corresponding pseudocodes. Section \ref{complexity} is 
devoted to the analysis of both, the algorithm 
complexity, and the numerical errors of the method. Finally in Section \ref{simulations} several benchmarks are run to assess the performance and scalability of the method, and whenever available, a comparison with the performance obtained by the classical Krylov-based method is done. In closing, we highlight the main results and suggest further directions for future research.

\section{Mathematical description of the probabilistic method and multilevel Monte Carlo method}\label{theory}

In order to implement any multilevel Monte Carlo method it is mandatory to have a probabilistic representation of the solution. Thus, we describe next the probabilistic method used so far to compute the action of a matrix exponential over a vector.

\subsection{Probabilistic method}

The probabilistic representation for the action of a matrix exponential over a vector was introduced in \cite{Acebron_SISC} for dealing exclusively with adjacency matrices 
of undirected graphs. However,  in the following we show that this representation can be straightforwardly generalized for dealing with arbitrary matrices.

Consider $A=\{a_{ij}\}_{i,j=1,\ldots,n}$ a general {\it n}-by-{\it n} matrix, $u$ a given {\it n}-dimensional vector, and $x$ an {\it n}-dimensional vector. This vector corresponds to
the vector solution after computing the action of a matrix exponential over the vector $u$, that is $x=e^{\beta A}\,u$. Here the parameter $\beta$ is a constant, typically interpreted as the time variable in partial 
differential equations, or an effective {\it "temperature"} of the network in problems related with complex networks (see \cite{Estrada_review}, e.g.).

Let us define a diagonal matrix $D$, represented hereafter as a vector ${\bf d}$, with entries $d_{ij}=0$ $\forall i\ne j$, $d_{ii}=d_i=a_{ii}+L_{ii}, i=1,\ldots,n$, and a matrix $T$ with entries $t_{ij}$ given by
\begin{equation}
t_{ij}=
    \begin{cases}
      L_{ii}, & \text{if}\ i=j \\
     (-1)^{\sigma_{ij}}L_{ij},  & \text{otherwise}
    \end{cases}
\end{equation}
where $\sigma=\{\sigma_{ij}\}$ is a binary matrix with entries taking the value $1$ when $a_{ij}< 0$, and $0$ otherwise. 
Here $L_{ij}$ denotes the Laplacian matrix, defined in the broad sense as a matrix with nonpositive off-diagonal entries $L_{ij}=-|a_{ij}|$, and zero row sums, that is 
$L_{ii}=-\sum_{j\ne i} L_{ij}$. Then, it holds 
that $A=D-T$. Note that our definition differs from the classical one $A=D-L$ addressed to adjacency matrices \cite{Merris}. Instead, in this paper matrix 
$A$ can be any matrix. This is possible due to 
two changes. First, our diagonal matrix $D$ is not a degree matrix since the diagonal term 
$a_{ii}$ in the original matrix is added to the degree of the row (stored in $L_{ii}$). Second, we replace matrix $L$ with matrix $T$, 
which takes into account that matrix $A$ can have both positive and negative values unlike an adjacency matrix. Thus, this does not constitute
any restriction in the class of matrices amenable to be represented probabilistically. Quite the contrary, one can see that any arbitrary 
matrix can be straightforward decomposed in such a way.

Finding a probabilistic representation for this problem requires in practice \cite{Acebron_SISC} to use a splitting method for approximating the action of the matrix exponential over the vector $u$ as follows,
\begin{equation}
\bar{x}=\left(e^{\Delta t D/2}e^{-\Delta t T} e^{\Delta t D/2}\right)^N\,u, \label{eq_general1}
\end{equation}
where $\Delta t=\beta/N$, which in the following and for convenience it will be termed as the time step. Note that $\bar{x}$ corresponds to an 
approximation of the true solution $x$. In fact, this corresponds to the Strang splitting method, and therefore leads to an 
error, which after one time step is known \cite{Jahnke} to be of order $\mathcal{O}(\Delta t^3)$ locally, and of order $\mathcal{O}(\Delta t^2)$ globally. Therefore, the true solution is recovered in the limit $N\to\infty$. The probabilistic representation for computing a single entry $i$ 
of the vector $\bar{x}$ is then given by
\begin{equation}
\bar{x}_i=e^{\Delta t\,d_i/2}\mathbf{E}[\prod_{k=1}^N \eta_k],
\label{prob_single_entry}
\end{equation}
where $\eta_k= \phi_k\, e^{\Delta t\,d_{i_k}}$, $k=1,\ldots,N-1$, and $\eta_N=\phi_N\, e^{\Delta t\,d_{i_N}/2}\,u_{i_N}$. The $i_k$, $k=1,\ldots,N$, is a sequence of $N$ discrete random variables 
with outcomes on $S=\{1,2,\cdots,n\}$, and $\phi_k$ a two-point random variable taking values $-1$ and $1$ with a probability related to the matrix $\sigma$. 
The probabilities $p_{i_{k-1}\,i_k}(t)$, $k=2,\ldots,N$, and $p_{i\,i_1}(t)$ for $k=1$,  correspond to the transition probabilities of a continuous-time Markov
chain generated by the infinitesimal generator $Q=-L$ and evaluated at time $\Delta t$ for each $k$, being solution of the Kolmogorov's backward equations \cite{Waymire},
\begin{equation}
P'(t)=Q\,P(t),\quad P(0)=\mathbbm{1}\quad \quad (t\geq 0),\label{backward}
\end{equation}
for the matrix transition probability $P=(p_{ij})$.

A neat picture of this probabilistic representation can be described as follows: A random path starting at the chosen entry $i$ is generated according to the continuous-time
Markov chain governed by the generator $Q$, and evolves in time jumping randomly from $i$ to any state on $S$. Along this process, $N$ functions $\eta_k$ are evaluated, and the 
solution is obtained through an expected value of a suitable multiplicative functional. 

Note that such a representation allows in practice to compute a single entry $i$ of the vector solution, but can be conveniently modified to represent the 
full vector solution $\bar{x}$. The probabilistic representation for computing a single entry of the vector solution requires generating suitable random paths 
evolving backward in time from the state $i$ at $t=\Delta t$ to a final state on $S$ for $t=0$. Instead, the probabilistic representation for the full vector
requires generating a random path that starts at a given state according to a specific initial distribution, and evolves forward in time governed by a continuous-time Markov chain generated by the transpose of the generator $Q$. The contribution to the entry $i$ of the vector is mathematically formalized through the following representation:
\begin{equation}
\hat{x}_i=e^{\Delta t\,d_i/2}U\,\mathbf{E}[\prod_{k=1}^N \eta_k ],
\label{prob_full_vector}
\end{equation}
where $\eta_k= \phi_k\, e^{\Delta t\,d_{i_k}}$, $k=1,\ldots,N-1$, and $\eta_N=\phi_N\, e^{\Delta t\,d_{i_N}/2}\, u_{i_N}$, and $U=\sum_{l=1}^n u_{l}$.

In order to adapt this representation to the multilevel Monte Carlo framework, it is convenient to use the typical notation used so far in the literature. This entails rewriting the probabilistic representation for computing a single entry $i$ of the vector solution as 
\begin{equation}
\bar{x}_i=\mathbf{E}[P], \quad P=\prod_{j=1}^{N/2} \eta^{(j)},\label{prob_level}
\end{equation}
where
\begin{eqnarray}
\eta^{(j)} &=&\phi^{(j)}\, e^{\Delta t(d_{i_k}/2+d_{i_{k+1}}+d_{i_{k+2}}/2)},\,\quad j=1,\ldots,N/2-1,\nonumber\\
\eta^{(j)} &=&\phi^{(j)}\,e^{\Delta t(d_{i_k}/2+d_{i_{k+1}}+d_{i_{k+2}}/2)}\,u_k,\, j=N/2.
\end{eqnarray}
Here $k=2j-1$, and $i_1=i$. A similar expression can be readily found for the probabilistic representation of the full vector solution. 

It is worth observing that we have to deal with two sources of error when implementing in practice the probabilistic method, that is the statistical error coming from the use of a finite sample size for estimating the expected value, and the error due to the splitting method. In fact this error can be considered as being the equivalent to the truncation error appearing in discretizing differential equations, and in the following it will be termed as truncation error.

\subsection{The multilevel Monte Carlo method}
The multilevel Monte Carlo method we have developed is essentially based on the well-known method many times described in the literature. In the following we introduce briefly the ideas underlying the method for those readers not familiar with the topic. For further details see the excellent survey in \cite{Giles1}, and references therein. 

Essentially, the goal of the geometric multilevel Monte Carlo method consists in approximating the finest solution $P_L$, obtained to the level
of discretization $L$, using a sequence of coarser approximations obtained at previous levels $l$, from $l_0$ to $L-1$. In our specific problem this corresponds to different levels of discretization according to the value of $\Delta t$, being now $\Delta t_l=\beta/N_l$, with $N_l=2^l$.
The minimum and initial level $l_0$ is chosen typically to be the entire interval, that is $\Delta t_0=\beta$. However this is not theoretically required, 
and for this specific problem we show in Section \ref{complexity} that it is best not to do so. This is because the computational cost tends to be independent of the 
level when simulating for the coarsest level of simulation. Therefore, in the following we assume that the minimum level to be chosen is $l_0$.
The multilevel method can be formalized mathematically through the following telescoping series,
\begin{equation}
\bar{x}^L=\mathbf{E}[P_L]=\mathbf{E}[P_{l_0}]+\sum_{l=l_0+1}^L m_l,\label{series_telescopic}
\end{equation}
where $m_l=\mathbf{E}[P_l-P_{l-1}]$, $ P_l=\prod_{j=1}^{N/2} \eta^{(j)}_l$, and 
\begin{eqnarray}
\eta^{(j)}_l &=&\phi ^{(j)}_l\, e^{\Delta t_l(d_{i_k}/2+d_{i_{k+1}}+d_{i_{k+2}}/2)},\,j=1,\ldots,N/2-1,\nonumber\\
\eta^{(j)}_l &=&\phi ^{(j)}_l\,e^{\Delta t_l(d_{i_k}/2+d_{i_{k+1}}+d_{i_{k+2}}/2)}\,u_k,\, j=N/2,\label{prob_level_l}
\end{eqnarray}
with $k=2j-1$, and $i_1=i$.
Note that this induces a truncation error which is proportional to $\mathbf{E}[P_L-P_{L-1}]$.
Numerically, when a finite sample of sizes $M_l, l=l_0,\ldots,L$ is used, Eq. (\ref{series_telescopic}) 
can be approximated by the following estimator
\begin{equation}
\bar{x}^L\approx\frac{1}{M_0}\sum_{i=1}^{M_0}\,P^{(i)}_{l_0}+\sum_{l=l_0+1}^L\frac{1}{M_l}\sum_{i=1}^{M_l}\,(P^{(i)}_{l}-P^{(i)}_{l-1}).
\label{mlmc_estimate}
\end{equation}
It is worth observing that the samples used for computing the approximation at level $l$ are reused for computing the level $l-1$ adapting them 
conveniently for such a coarse level. In fact, the underlying correlation appearing between the two consecutive levels belonging to the same sample 
becomes essential in order to reduce the overall variance for the same computational cost. However, the final goal of the multilevel method is the opposite, that is, reducing the computational cost by choosing
conveniently an optimal sample size $M_l$, keeping fixed the overall variance within a prescribed accuracy $\varepsilon^2$. After a suitable minimization process, the result as explained in \cite{Giles1} is given by
\begin{equation}
M_l=\frac{1}{\varepsilon^2}\sqrt{\frac{V_l}{C_l}}\sum_{l=l_0}^{L}V_l C_l,
\end{equation}
where $C_l$, and $V_l$ are the computational cost, and the variance for each level $l$, respectively. The overall computational cost and variance can be calculated as follows
\begin{equation}
C_T=\sum_{l=l_0}^{L} C_l,\quad\quad V_T=\sum_{l=l_0}^{L} \frac{V_l}{M_l},\label{vl_ml}
\end{equation}

\section{The multilevel algorithm}\label{algorithm}

To implement in practice the multilevel method for computing the action of the matrix exponential over a vector, it is first necessary to introduce a suitable algorithm capable of generating
efficiently the random paths. Second, we need to describe the strategy followed to compute the difference between any two consecutive levels as appears in Eq. (\ref{mlmc_estimate}). This requires an
efficient technique to reuse the paths obtained when simulating with a higher level $l$ for the lower level at discretization $l-1$. 

Concerning the first issue,  we describe next the numerical method proposed to generate in practice the continuous-time Markov chain.
Let $p_{ij}(t)$ represent the transition probability matrix. Then the Kolmogorov's backward equation in Eq. (\ref{backward}) can be equivalently represented as the 
following system of integral equations \cite{Waymire} 
\begin{equation}
p_{ij}(t)=\delta_{ij}\,e^{-L_{ii} t}+\sum_{j\neq i}\int_0^t ds\, L_{ii}\, e^{-L_{ii}\, s} k_{ij} p_{ij} (t-s),\label{backward2}
\end{equation}
where $k_{ij}=|L_{ij}|/L_{ii}$. Let $S_0,S_1,\ldots$ be a sequence of independent exponential random times picked up from the exponential probability density $p(S_i)=
L_{ii}\,e^{-L_{ii} S_i}$. The integral equations above along with the sequences of random times can be used to simulate a path according to the following recursive algorithm: Generate a first random time $S_0$ obeying the
exponential density function; Then, depending on whether $S_0 < t$ or not, two
different alternatives are taken; If $S_0 > t$, the algorithm is stopped, and no jump from the state $i$ to a different state is taken;
If, on the contrary, $S_0 < t$, then the state $i$ jumps to a different state $j$ according to the probability function $k_{ij}$, and a new second random number exponentially distributed $S_1$
is generated; If $S_1 < (t-S_0)$ the algorithm proceeds repeating the same elementary rules, otherwise it is stopped.

\begin{figure}[!t]
\includegraphics[width=4.5in,angle=-90]{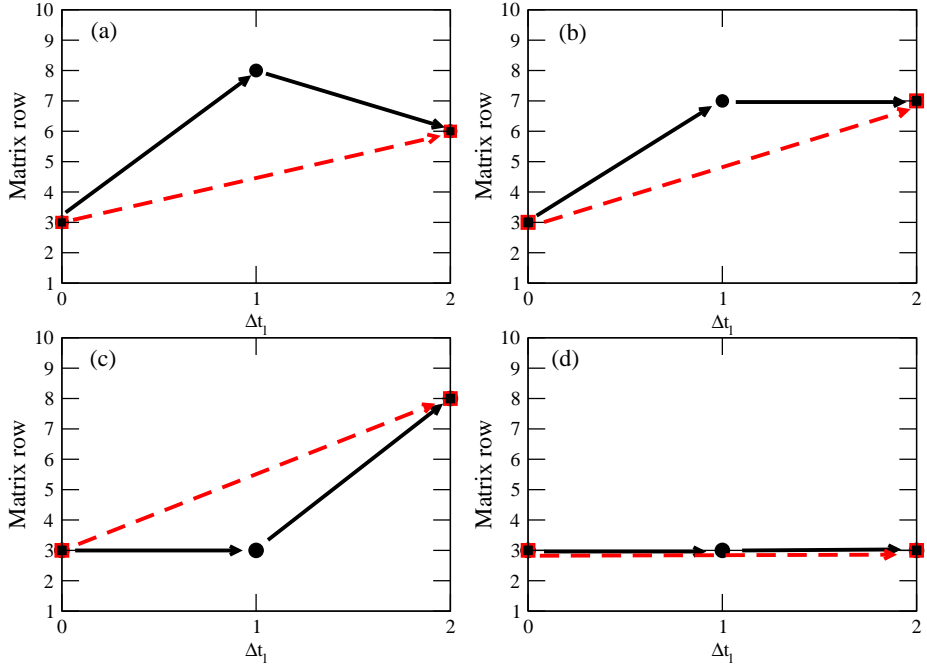}
\caption{Sketch diagram showing the four possible sampled paths obtained for level $l=2$, and for a matrix of size $n=10$. The solid 
line corresponds to a random path obtained for a level number $l$, and the dotted line with $l-1$.}
\label{path}
\end{figure}

In Fig. \ref{path} a sketch diagram for the case of $l=2$ is shown. This illustrates graphically how the second issue, related to the computation of the coarse level $l-1$ using the higher level $l$, has been solved in practice. There we plot the four different scenarios that may occur when generating random paths (assuming we are interested in computing only a single entry $i$ of the 
vector solution, and therefore forcing all random paths to start at the same state $i$). Thus, from Eq. (\ref{prob_level}), the possible outcomes of the two 
random variables may induce two transitions to any of the rows of a given matrix during the two time steps of size $\Delta t_2$. But only the last one should be used for determining the paths corresponding to the previous level $l=1$. More specifically, the set of the four figures describe the following scenarios: a) Transitions occur at the first and the second time step; b) Transition only at the first time step; c) Transition only at the second time step, and d) no transition at all. Note that the last scenario contributes with zero to the term $\mathbf{E}[P_2-P_{1}]$ in (\ref{series_telescopic}).

In Algorithm \ref{pseudocode_main}, we describe a pseudocode corresponding to the implementation of the multilevel method. In fact, this consists in the general setting for 
any implementation of the method for a variety of problems. The distinguishing feature among them is the suitable procedure chosen to compute in practice any of the terms of the expansion in Eq. (\ref{series_telescopic}), as well as the associated variances. The pseudocode of the procedure for computing a single entry of the vector solution is described in Algorithm \ref{pseudocode1}.

Although the multilevel method could be used to compute the full vector solution as well, the implementation is much more involved and the performance of the algorithm 
less efficient. This is because it will require in practice to save vectors instead of scalars for any of the levels in  Eq. (\ref{series_telescopic}). This can be mitigated instead by
computing a scalar function of the full vector solution, and since the complexity of the algorithm for computing the full vector solution by Monte Carlo is similar to that for obtaining the solution of a single entry, in principle the computation time of the multilevel method for the former case should be comparable. In fact, the pseudocode is similar (see Eq. \ref{prob_single_entry} and Eq. \ref{prob_full_vector}).


\begin{algorithm}
\caption{Multilevel Monte Carlo (MLMC) algorithm.}\label{pseudocode_main}
\begin{algorithmic}
\State INPUT: $L=l_0+4$, $M=M_0$, $i, N$, $\varepsilon$, $\beta$
\State Call {\bf MLMCL}($i,\Delta t_l,N,M_0$) for fast estimating $m_l$ and $V_l$ for $l=l_0,\ldots,L$
\While{$error\ge\varepsilon$}
\State Compute the optimal number of samples $M_l$ for $l=l_0,\ldots,L$
\State Call {\bf MLMCL}($i,\Delta t_l,N,M_l$) for further improvement for $l=l_0,\ldots,L$
\If {$error \le \varepsilon$} 
EXIT
\Else
\State Increase number of levels, $L=L+1$
\EndIf
\EndWhile

\end{algorithmic}
\end{algorithm}

\begin{algorithm}
\caption{Procedure to compute a single entry $i$ of the vector solution $\bar{x}_i$.}\label{pseudocode1}
\begin{algorithmic}
\Procedure{MLMCL}{$i,\Delta t_l,N,M$}
\State $m_l=0$, $m2l=0$
\For {$l=1,M$}
\State $\eta_1=1$, $\eta_2=1$, $j=i$
\For {$n=1,\ldots,N$}
\State $\eta_2=\eta_2 e^{d_{j} \Delta t_l/2} $
\If {$n\, mod\, 2\neq 0$}
\State $\eta_1=\eta_1 e^{d_{j} \Delta t_l} $
\EndIf
\State generate $\tau$ exponentially distributed 
\While{$\tau<\Delta t_l$}
\State generate $S$ exponentially distributed 
\State $k=j$
\State generate $j$ according to Eq.(\ref{backward2}) 
\State $\tau=\tau+S$
\State $\eta_2=(-1)^{\sigma_{kj}}\eta_2$
\State $\eta_1=(-1)^{\sigma_{kj}}\eta_1$
\EndWhile
\State $\eta_2=\eta_2 e^{d_{j} \Delta t_l/2} $
\If {$n\, mod\, 2= 0$}
\State $\eta_1=\eta_1 e^{d_{j} \Delta t_l} $
\EndIf

\EndFor
\State $m_l=m_l+[u_j (\eta_2-\eta_1)]/M$
\State $m2l=m2l+[u_j (\eta_2-\eta_1)]^2/M$
\EndFor
\State $V_l=m2l/M-m_l^2$

\hspace{-0.35cm} \Return $(m_l,V_l)$
\EndProcedure
\end{algorithmic}
\end{algorithm}

\section{Convergence and Computational complexity of the multilevel algorithm}\label{complexity}

The computational complexity of any MLMC algorithm can be established properly resorting to Theorem 1 in \cite{Giles1}. However, it is mandatory to characterize previously
the convergence of some important quantities such as the mean $|E[P_l-P]|$ and variance $V[P_l-P_{l-1}]$, as well as the computational time of the Monte Carlo algorithm, as a function of 
the level $l$. 

Concerning the scaling of the mean $|E[P_l-P]|$ with the level $l$, it can be readily estimated as follows.
Since $E[P]$ corresponds to the theoretical solution, $x= e^{\beta A}\,u$, obtained probabilistically in practice when $N\to\infty$, $|E[P_l-P]|$ corresponds in fact to the truncation error $|E[P_l]-x]|$. Recall that this was considered previously as being due to the Strang splitting method. Therefore,  the local error after one time step $\varepsilon_S$ of this approximation 
is known \cite{Jahnke} to be
\begin{equation}
\varepsilon_S=\Delta t_l^3(\frac{1}{12} [D,[D,T]]-\frac{1}{24} [T,[T,D]])\,u+\mathcal{O}(\Delta t_l^4),\label{errorStrang}
\end{equation}
and globally of order $\mathcal{O}(\Delta t_l^2)$. This is in agreement with Fig. \ref{fig_mlVl}(a), where the mean $|E[P_l-P_{l-1}]|$ is plotted as a function of the level $l$ for the example 
consisting in simulations of a small-world network of three different sizes.

Characterizing the variance $V[P_l-P_{l-1}]$ as a function of the level $l$ turns out to be a much more involved procedure. To start, it holds that
\begin{equation}
V[P_l-P_{l-1}]=E[(P_l-P_{l-1})^2]-(E[P_l-P_{l-1}])^2\leq E[(P_l-P_{l-1})^2].
\end{equation}
Hence, the problem can be reduced to the problem of estimating $E[(P_l-P_{l-1})^2]$. For this purpose, and as a preliminary step, it will be estimated next a partial result regarding the random 
variable  $\eta^{(2)}_l$ and $\eta^{(1)}_{l-1}$, and then the final result will be estimated accordingly. These random variables are obtained when generating paths for a single 
time step (when the  level is $l-1$),  and two consecutive time steps (when the level is $l$). The superscripts $(2)$ and $(1)$ denote two and one consecutive
steps respectively. Therefore, we first establish the following Lemma. 
\begin{lemma}\label{Lemma1}
Let $j$ and $k$ discrete random variables that take values on  $\Omega=\{1,2\cdots,n\}$, with probability $p_{ij}(t)$ and $p_{jk}(t)$ given by the transition probabilities of a 
continuous-time Markov chain generated by the infinitesimal generator $Q=-(L)_{ij}$ and evaluated at time $\Delta t_l$. Then, it holds that 
\begin{equation}
\mathbf{E}[(\eta^{(2)}_l-\eta^{(1)}_{l-1})^2]=\mathcal{O}(\Delta t_l^3),
\end{equation}
where $\eta^{(2)}_l=e^{\Delta t_l\,d_i/2}  e^{\Delta t_l\,d_j} e^{\Delta t_l\,d_k/2}\,u_k$, and 
$\eta^{(1)}_{l-1}=e^{\Delta t_l\,d_i} e^{\Delta t_l\,d_k}\,u_k$, respectively. 
\end{lemma}
\begin{proof}
Expanding $\eta^{(2)}_l$ and $\eta^{(1)}_{l-1}$ in powers of $\Delta t_l$ yields
\begin{equation}
\mathbf{E}[(\eta^{(2)}_l-\eta^{(1)}_{l-1})^2]=\Delta t_l^2 \mathbf{E}[\xi^2]+\mathbf{E}[\mathcal{O}(\Delta t_l^3)],
\end{equation}
where $\xi=(-d_i/2+d_j-d_k/2)u_k$. The possible outcomes of the random variables $j$, and $k$ can be one of the following four different cases: (a) $j\neq k\neq i$; 
(b) $j\neq i,k=j$; (c) $j=i,k\neq i$ and (d) $j=k=i$ (see Fig. \ref{path} for illustration). To distinguish among them, consider one pair of binary variables 
$(\alpha_1, \alpha_2)$, taking values 
$\{(0,0),(0,1),(1,0),(1,1)\}$ and corresponding to the cases $a,b,c$, and $d$, respectively.
From the transition probabilities of the corresponding continuous-time Markov chain, the probability of obtaining each of them is given by
\begin{equation}
p_{\alpha_1\alpha_2}=[\alpha_1\,e^{\Delta t_l\,d_i}+(1-\alpha_1)(1-\,e^{\Delta t_l\,d_i})][\alpha_2\,e^{\Delta t_l\,d_j}+(1-\alpha_2)(1-\,e^{\Delta t_l\,d_j})]
\end{equation}
By expanding $p_{\alpha_1\alpha_2}$ in powers of $\Delta t_l$, we have
\begin{equation}
p_{\alpha_1\alpha_2}=\alpha_1\,\alpha_2 +C\,\Delta t_l+\mathcal{O}(\Delta t_l^2),
\end{equation}
where $C$ depends merely on $d_i,d_j,d_k$. Note that for the case (d), which corresponds to $(\alpha_1,\alpha_2)=(1,1)$, $\xi$ turns out to be $0$, therefore it follows that 
$\mathbf{E}[(\eta^{(2)}_l-\eta^{(1)}_{l-1})^2]$ is $\mathcal{O}(\Delta t_l^3)$ and the proof is complete.
\end{proof}

\begin{figure}[!t]
\hspace{-1cm}
\includegraphics[width=2.5in,angle=-90]{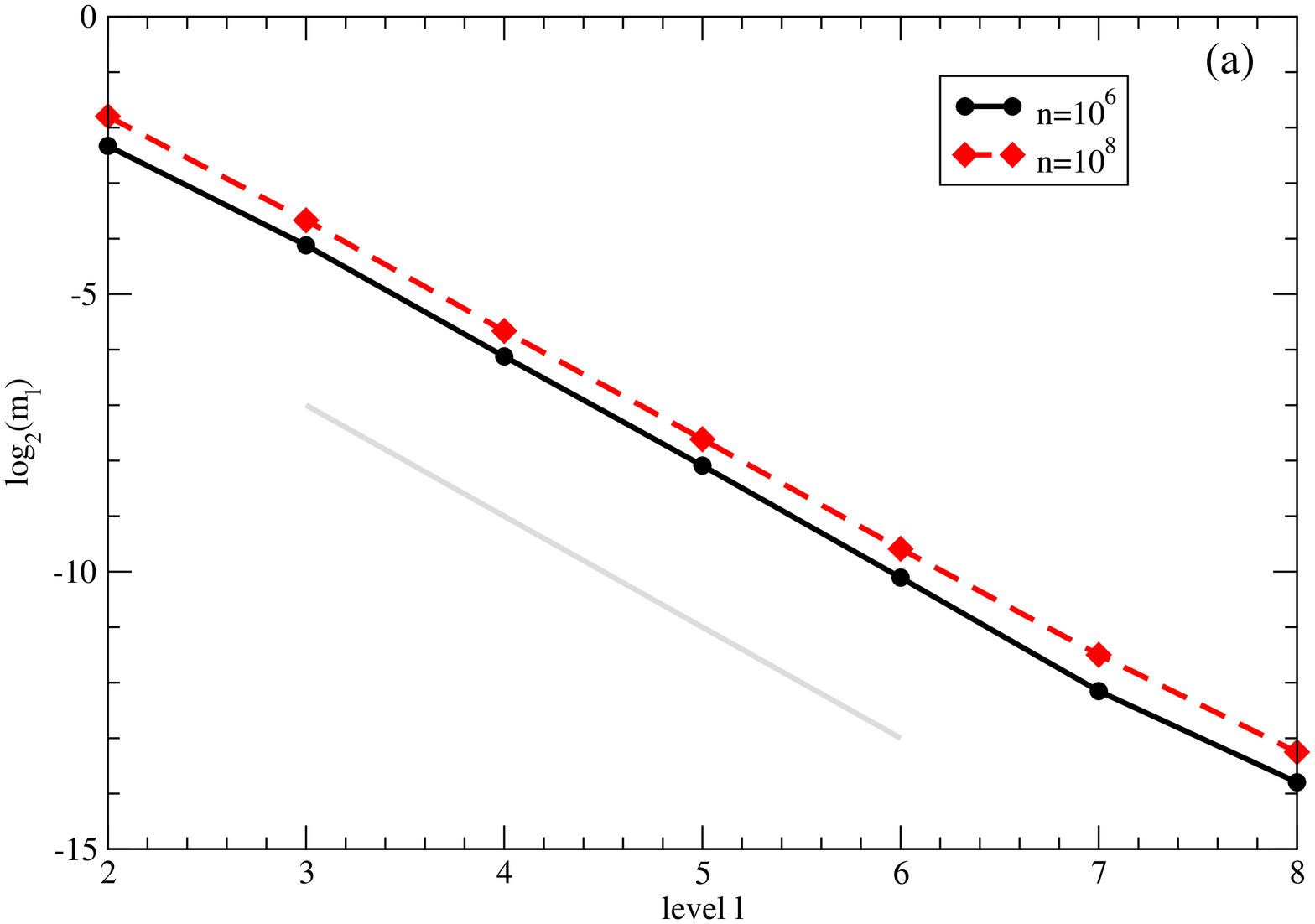}
\includegraphics[width=2.5in,angle=-90]{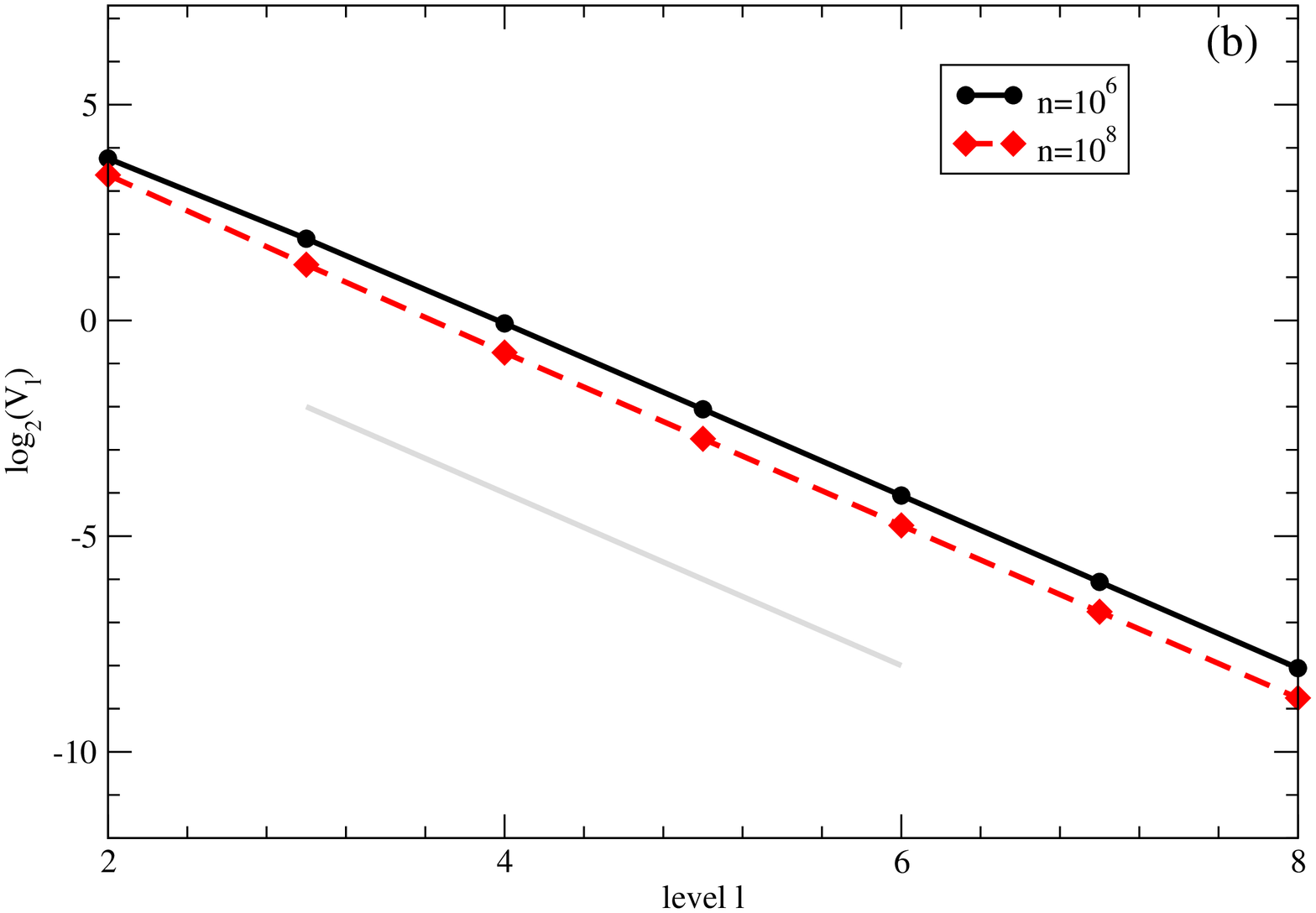}
\caption{(a) Mean $m_l$ and (b) variance $V_l$ of $P_l-P_{l-1}$ in $log_2$ scale versus the level number $l$ obtained numerically. The adjacency matrix corresponds to a small-world network of different sizes $n$. The brown line denotes an ancillary function of slope $-2$.}
\label{fig_mlVl}
\end{figure}

To find the global convergence rate, consider the following Lemma.

\begin{lemma}
Assume $i_k$, with $k=1,\cdots,N$, are $N$ discrete random variables taking values on $\Omega=\{1,2\cdots,n\}$, with probability $p_{i_k\, i_{k+1}}(t)$ given by the transition probabilities of a continuous-time Markov chain generated by the infinitesimal generator $Q=-(L)_{ij}$ and evaluated at time $\Delta t_l$. Then, it holds that
\begin{equation}
\mathbf{E}[(P_l-P_{l-1})^2]=\mathcal{O}(\Delta t_l^2)
\end{equation}
\end{lemma}
\begin{proof}
Expanding  $\eta^{(j)}_l$ and $ \eta^{(j)}_{l-1}$ in Eq. (\ref{prob_level_l}) in powers of $\Delta t_l$, we have
\begin{equation}
\mathbf{E}[(P_l-P_{l-1})^2])=\Delta t_l^2 \mathbf{E}[(\sum_{j=1}^{N/2}\xi_j)^2]+\mathcal{O}(\Delta t_l^3),
\end{equation}
where $\xi_j=(d_{i_k}/2+d_{i_{k+1}}+d_{i_{k+2}}/2)-(d_{i_k}+d_{i_{k+2}})$, and $k=2j-1$. Taking into account that $\xi_j$ are independent random variables governed by the same probability
function, it follows that
\begin{equation}
\mathbf{E}[(\sum_{j=1}^{N/2}\xi_j)^2]=\sum_{j=1}^{N/2} \mathbf{E}[\xi_j^2]+2\sum_{j<k}\mathbf{E}[\xi_j]\mathbf{E}[\xi_k]. 
\end{equation}
Since $\mathbf{E}[d_{i_k}]$ is independent of $i_k$, then  $\mathbf{E}[\xi_j]=0$. Therefore, we have   
\begin{equation}
\sum_{j=1}^{N/2} \mathbf{E}[\xi_j^2]=\frac{N}{2}\mathbf{E}[\xi_1^2].
\end{equation}
Finally from Lemma \ref{Lemma1} it turns out that $\mathbf{E}[\xi_1^2]$ is of order $\mathcal{O}(\Delta t_l^3)$
and since $\Delta t_l=\beta/N_l$ by definition, then it follows that $\mathbf{E}[(P_l-P_{l-1})^2])=\mathcal{O}(\Delta t_l^2)$.
\end{proof}

In Fig. \ref{fig_mlVl}(b), $V[P_l-P_{l-1}]$ is shown as a function of the level $l$. The adjacency matrices correspond to a small-world network of three different sizes. Note that the obtained numerical convergence rate fully agrees with the theoretical estimation.

\begin{figure}[!t]
\includegraphics[width=4.5in,angle=-90]{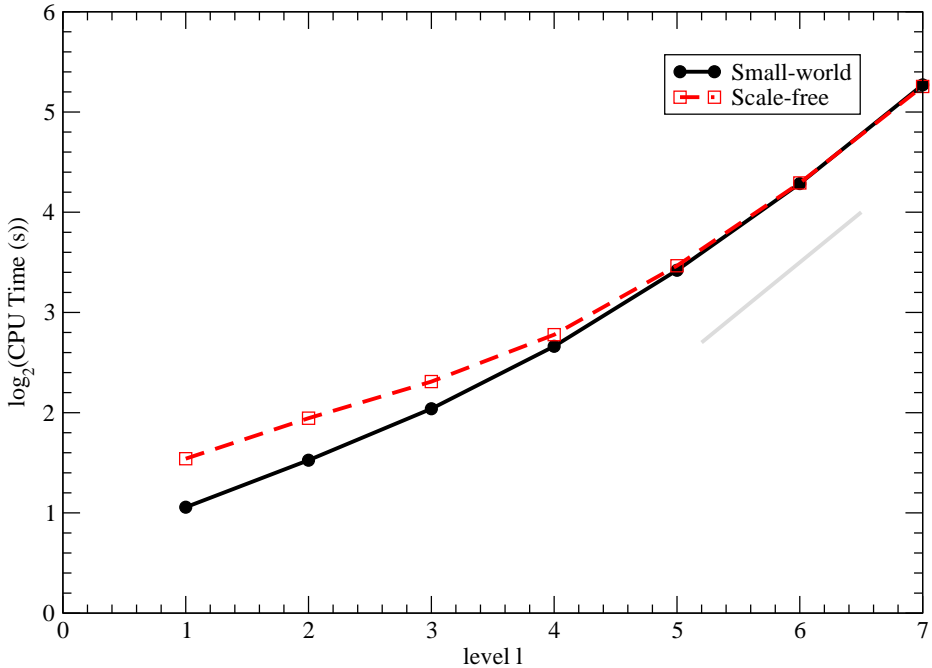}
\caption{Computational time in $log_2$ scale versus the level $l$ for adjacency matrices corresponding to two different complex networks of size $n=10^6$. The brown line corresponds to an ancillary function of slope $1$.}
\label{fig_cpu}
\end{figure}

The computational time of the Monte Carlo algorithm was already estimated in \cite{Acebron_SISC}, and it is given by
\begin{equation}
T_{CPU}=\alpha_{in}\beta \bar{d}\,M+\alpha_{out}\frac{\beta}{\Delta t_l}\,M. \label{tcpu}
\end{equation}
Here $\bar{d}$ is $\bar{d}=\frac{1}{n}\sum_{i=1}^n d_i$, while $\alpha_{in}$ and $\alpha_{out}$ are suitable proportionality constants. In Fig.  \ref{fig_cpu}, the results corresponding to the CPU time spent by the Monte Carlo algorithm when computing the total communicability of two different networks characterized by different values of $\bar{d}$ is shown. The results are in agreement with the theoretical estimation in Eq. (\ref{tcpu}).
In particular, note that for $\Delta t_l$ sufficiently large (or equivalently $l$ sufficiently small) the computational time tends to a constant value, while for smaller values the computational time scales as $1/\Delta t_l$. 
This also explains what was mentioned previously in Section \ref{theory}, which is that the initial level $l_0$ of the multilevel method could be different from zero to obtain a better performance of the MLMC algorithm.
In fact, depending on the value of $\Delta t_l$ and consequently on the level $l$, two different working regimes can be observed, and only for the regime characterized by a 
value of $\Delta t_l$ sufficiently small, the computational time asymptotically increases with $l$. Specifically this occurs when the contribution to the computational time of the second 
term in Eq. (\ref{tcpu}) is much larger than the first term. Assuming that the value of the proportionality constants $\alpha_{in}$, $\alpha_{out}$ are similar, we can readily estimate 
the minimum value of the level needed for this purpose, and is given by
\begin{equation}
l_0 \gg log_2{(\beta \bar{d})}.
\end{equation}
However, in general both constants $\alpha_{in}$, and $\alpha_{out}$ are not only different, but also difficult to be theoretically estimated. From numerical simulations, however, a more practical lower bound has been found and is given by
\begin{equation}
l_0=log_2{(2 \beta d_{max})},\label{lowerbound}
\end{equation}
where $d_{max}$ corresponds to the maximum value of the diagonal matrix $D$.  In Fig. \ref{fig_cpu2} the computational time spent by the MLMC method for different values of the initial level $l_0$ is shown. Here the MLMC was applied to the problem of computing the total communicability \cite{Benzi2} of two different networks, small-world and scale-free, of size $n=10^6$. For the small-world network the maximum degree is $7$, while for the scale-free 
network is $3763$. The value of $\beta$ was chosen to be $1$ for the small-world network and $1/d_{max}$ for the scale-free network. The last one was chosen 
specifically to ensure the convergence of the method, as it was pointed out in \cite{Acebron_SISC}. Note that, for both networks, the computational time attains a
minimum at a specific value of $l_0$, which is well approximated by Eq. ({\ref{lowerbound}}).

\begin{figure}[!t]
\includegraphics[width=4.5in,angle=-90]{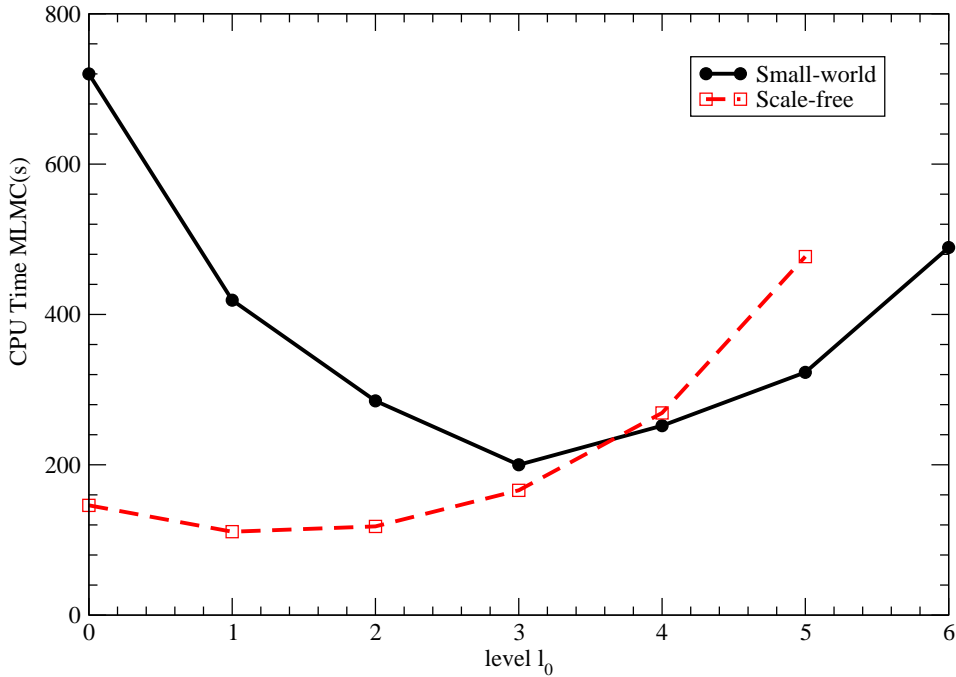}
\caption{Computational time of the MLMC method for different values of the initial level $l_0$. The matrices correspond to the adjacency matrices of two different complex networks of size $n=10^6$. The accuracy for the small-world network is $\varepsilon=6.25\times10^{-4}$, while for the scale-free network is $\varepsilon=2\times10^{-7}$.}
\label{fig_cpu2}
\end{figure}

\begin{figure}[!t]
\includegraphics[width=4.5in,angle=-90]{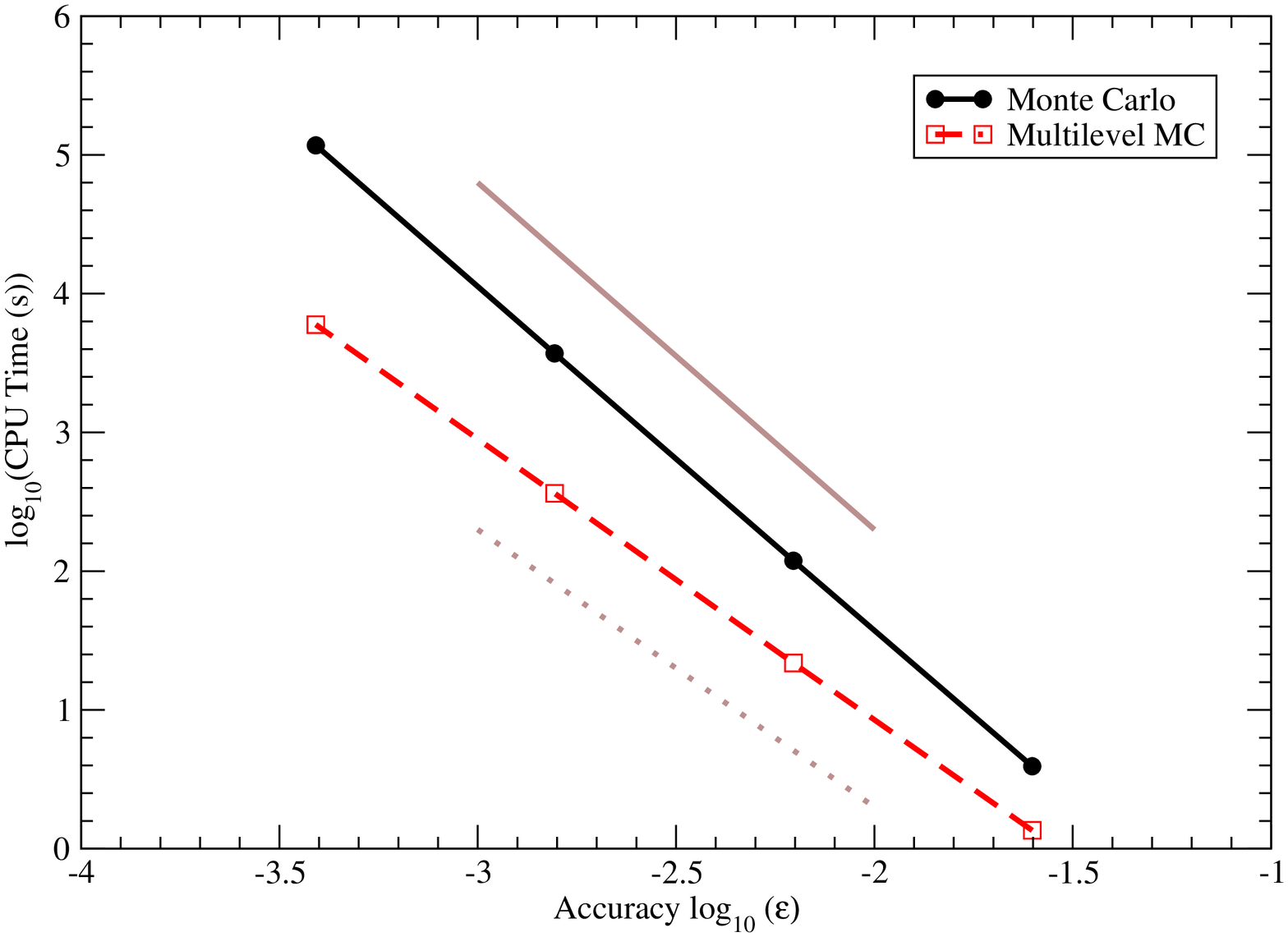}
\caption{Computational time as a function of the prescribed accuracy $\varepsilon$, both in $log_{10}$ scale. The brown solid line corresponds to an ancillary function of slope $-5/2$,
while the dotted line to a function of a slope $-2$. The results correspond to the communicability for a single node of a small-world network of size $n=10^6$}
\label{fig_complexity}
\end{figure}

In view of the convergence rates estimated above, we can apply the aforementioned Theorem 1 in \cite{Giles1} and conclude that the computational complexity of the proposed MLMC 
algorithm is of order $\mathcal{O}(\varepsilon^{-2})$.  In practice this means that the error due to the splitting in Eq. (\ref{eq_general1}) is totally canceled out from the algorithm, 
remaining only the computational cost inherent to any Monte Carlo method due to the statistical error. Rather, the complexity of the classical Monte Carlo algorithm proposed in \cite{Acebron_SISC} is of order of  $\varepsilon^{-5/2}$. Indeed 
this can be readily proved as follows. Concerning the statistical error, the sample size $M$ required to achieve a prescribed accuracy $\varepsilon$ is given by $M=\varepsilon^{-2}$, while for 
the splitting error, being the method of order of $\Delta t^2$, the time step required for a given $\varepsilon$ is $\Delta t=\varepsilon^{1/2}$. Therefore, the computational complexity, which 
depends on $M/\Delta t$, is given by  $\mathcal{O}(\varepsilon^{-5/2})$. In Fig. \ref{fig_complexity} the results corresponding to the computational time spent to 
compute the communicability of a single node of a small-world
network of size $n=10^6$ are plotted as a function of a chosen prescribed accuracy $\varepsilon$ for both, the multilevel Monte Carlo and the classical Monte Carlo method.  Note the perfect
agreement with the theoretical estimates, and the performance notably superior to the classical Monte Carlo method in \cite{Acebron_SISC} for lower accuracy values.

\section {Performance evaluation} \label{simulations}

To illustrate the performance of the multilevel Monte Carlo method, in the following we show the results corresponding to several benchmarks conducted so far. They concern  the numerical solution of a linear parabolic differential equation by means of an exponential integrator, as well as, the numerical computation of the total communicability metric in some complex synthetic networks. 

In fact, among an important application of the multilevel algorithm for computing the action of a matrix exponential, we have the numerical solution
of parabolic PDEs by means of the method of lines, using therefore an exponential integrator. When the method of lines \cite{Mazunder} is applied to an initial parabolic PDE problem discretizing the spatial variable, a system of coupled ordinary differential equations, with time as the independent variable, is obtained. Finally, the system can be solved resorting 
to the computation of a matrix exponential which acts on the discretized initial value function. The method was here applied for solving the Dirichlet boundary-value problem for both, a 3D heat equation, and a 3D convection-diffusion equation. The former problem is given by
\begin{eqnarray}
&& \frac{\partial u}{\partial t} = \nabla^2 u,
      \quad    \mbox{in} \
                      \Omega = [-\delta, \delta]^3, t > 0, 
\end{eqnarray}
with boundary- and initial-conditions
\begin{eqnarray}
\left. u({\bf x},t) \right|_{\partial \Omega} = 0, 
            \qquad
   u({\bf x},0) = f({\bf x}).
\end{eqnarray}
The approximated solution ${\hat u}(\bf x,t)$ where ${\bf x} \in {\bf R}^3$, ${\bf x}=(x,y,z)$, after discretizing in space with grid 
spacing $\Delta x=\Delta y=\Delta z=2\delta/n_x$, and using the standard $7-$point stencil finite difference approximation, can be written formally as
\begin{equation}
{\hat u}({\bf x},t)=e^{\frac{n_x^2 t}{4 \delta^2}{\hat L}}\,{\hat u}_0({\bf x}),\label{pde_lines}
\end{equation}
where ${\hat L}$ denotes the corresponding discretized Laplacian operator.


Concerning the convection-diffusion equation, mathematically we have

\begin{eqnarray}
&& \frac{\partial u}{\partial t} = \nabla^2 u+{\bf \beta}\cdot\nabla u,
      \quad    \
                      {\bf x}\in \Omega, t > 0,\nonumber \\
                 &&     \left.u({\bf x},t) \right|_{\partial \Omega} = g({\bf x,t}),\\
                  &&    u({\bf x},0) =f({\bf x}),\nonumber 
\end{eqnarray}
where ${\bf \beta}$ is the velocity field. After applying the standard Galerkin finite element method \cite{Zienkiewicz} to the discretized nodes ${\bf x}_i, i=1,\ldots,n$, the following
linear system of coupled first order ODEs is obtained
\begin{equation}
M\frac{d {\bf u}}{d t} = K{\bf u}+{\bf F},\quad {\bf u}(0)={\bf u}_0,\label{ODEs}
\end{equation}
where ${\bf u}=(u({\bf x}_1,t),\ldots,u({\bf x}_n,t))$, $M$ is the assembled mass matrix, $K$ is the corresponding assembled stiffness matrix, and ${\bf F}$ is 
the load vector. Concerning the boundary data, these are included modifying as usual the matrices and the vector.  For computational convenience, in the following the mass matrix was lumped \cite{Zienkiewicz}, resulting in practice in a diagonal mass matrix.

Formally, the solution of the inhomogeneous system of ODEs (\ref{ODEs}) can be written in terms of a matrix exponential as follows
\begin{equation}
{\bf u}({\bf x},t)=e^{-t\,M^{-1}K}{\bf u}_0+ \int_0^t ds\, e^{-s\,M^{-1}K}{\bf F}(t-s).\label{integral_ODE}
\end{equation}
Note that for the particular case of having time-independent boundary data, the load vector becomes therefore constant, and the solution simplifies to 
\begin{equation}
{\bf u}({\bf x},t)=e^{-t\,M^{-1}K}{\bf u}_0-K^{-1}M\left(e^{-t\,M^{-1}K}-\mathbbm{1}\right){\bf F}.\label{solution_ODE}
\end{equation}
On the other hand, for arbitrary time-dependent boundary data, the integral in Eq. (\ref{integral_ODE}) can be computed resorting to suitable numerical quadratures. This procedure can be followed in any case to avoid evaluating the inverse of the matrix $K$ in Eq. (\ref{solution_ODE}), which in general can be computationally costly. In fact, this was specifically used here for solving 
numerically the system of equations in (\ref{ODEs}). Since to compute the matrix exponential the error was estimated to be of order $\mathcal{O}(\Delta t^2)$ (see Sec. \ref{complexity}), to
avoid lowering down this order, in the following we have implemented the Simpson quadrature rule, which is known to be of much higher order. More specifically, the solution ${\bf \hat u}({\bf x},t)$ can be computed as follows
\begin{eqnarray}
{\bf \hat u}({\bf x},t)=e^{-t\,M^{-1}K}{\bf u}_0+\Delta t\left({\bf F}(t)+2\sum_{j=1}^{N/2-1}e^{-t_j\,M^{-1}K}{\bf F}(t-t_j)\right.\nonumber
\\\left.+4\sum_{j=1}^{N/2}e^{-t_j\,M^{-1}K}{\bf F}(t-t_j)+e^{-t\,M^{-1}K}{\bf F}(0)\right),
\end{eqnarray}
where $t_j=j\Delta t,\,j=1,\ldots,N$, and $\Delta t=t/N$. Note that we require to compute $N$ independent matrix exponential evaluations at $N$ different instants of time. However, it turns out that using the Algorithm \ref{pseudocode1}, in practice only a single evaluation at the final time $t$ is needed to compute. This is because when using the Monte Carlo method for computing the matrix exponential at time $t$, the information required to evaluate the matrix exponential at intermediate times have been also automatically generated by the algorithm.  In fact, the random paths generated up to time $t$, which have been simulated advancing in time steps of size $\Delta t$, can be used directly to evaluate the matrix exponential over 
the vector ${\bf F}$ at time $j\Delta t$. Moreover, it is worth observing that this can be accomplished without any additional computational cost. 

In practice this can be readily done modifying slightly the Algorithm, as it is shown in boldface in the new Algorithm \ref{new_pseudocode}. Here 
$\omega_n$ is a vector containing the suitable weights $(1, 2, 4,\ldots, 2,4,1)$ corresponding to the Simpson quadrature rule.

\begin{algorithm}
\caption{Procedure to compute a single entry $i$ of the vector solution ${\bf \hat u}({\bf x},t)$}\label{new_pseudocode}
\begin{algorithmic}
\Procedure{MLMCL-FEM}{$i,\Delta t_l,N,M$}
\State $m_l=0$, $m2l=0$,$integ_1=0$,$integ_2=0$
\For {$l=1,M$}
\State $\eta_1=1$, $\eta_2=1$, $j=i$
\For {$n=1,\ldots,N$}
\State $\eta_2=\eta_2 e^{d_{j} \Delta t_l/2} $
\If {$n\, mod\, 2\neq 0$}
\State $\eta_1=\eta_1 e^{d_{j} \Delta t_l} $
\EndIf
\State generate $\tau$ exponentially distributed 
\While{$\tau<\Delta t_l$}
\State $k=j$
\State generate $S$ exponentially distributed 
\State generate $j$ according to Eq.(\ref{backward2}) 
\State $\tau=\tau+S$
\State $\eta_2=(-1)^{\sigma_{kj}}\eta_2$
\State $\eta_1=(-1)^{\sigma_{kj}}\eta_1$
\EndWhile
\State $\eta_2=\eta_2 e^{d_{j} \Delta t_l/2} $
\If {$n\, mod\, 2= 0$}
\State $\eta_1=\eta_1 e^{d_{j} \Delta t_l} $
\EndIf
\State ${\bf \color{blue} integ_1=integ_1+\omega_n \eta_1}$
\State ${\bf \color{blue} integ_2=integ_2+\omega_n \eta_2}$
\EndFor

\State ${\bf \color{blue} m_l=m_l+[u_j (\eta_2-\eta_1)]/M+[F_j (integ_2-integ_1)]/M}$
\State ${\bf \color{blue} m2l=m2l+[u_j (\eta_2-\eta_1)]^2/M+[F_j (integ_2-integ_1)]^2/M}$
\EndFor
\State $V_l=m2l/M-m_l^2$

\hspace{-0.35cm} \Return $(m_l,V_l)$
\EndProcedure
\end{algorithmic}
\end{algorithm}


Other important application of the matrix exponential consists in computing the total communicability of a network.  By definition, the total communicability of a network 
\cite{Benzi2} is given by 
\begin{equation}
TC=({\bf 1},\,e^A\,{\bf 1}),
\end{equation}
where $\bf{1}$ is a vector of ones, and $(\cdot,\cdot)$ denotes the scalar product. In the following we analyze the total communicability for several networks consisting in generated synthetic networks of the type small-world and scale-free of arbitrary size. These networks have been generated in Matlab using the functions 
{\it smallw} and {\it pref}, respectively, both freely available through the toolbox CONTEST \cite{CONTEST}. 
In contrast to the small-world network, the scale-free networks are characterized by the presence of hubs, which in practice entail a much larger 
maximum eigenvalue than for the small-world networks. Then, since the value of this eigenvalue increases with the network size, and in order to keep constant the numerical error, it may be necessary for the MLMC method to increase strongly the number of required levels accordingly. To prevent such a computationally costly procedure, a reasonable alternative relies on computing a generalization of the communicability, that is $e^{\beta A}$, where $\beta$ is typically interpreted as an effective "temperature" of the network (see \cite{Estrada_review}, e.g.). Essentially the idea that was exploited in \cite{Acebron_SISC} was to use the inverse of the maximum eigenvalue as the value of the parameter $\beta$, which in practice will control the rapid growth of the norm of the matrix $A$ with the size of the network. Different values of $\beta$ could have a direct impact not only on the entries of the communicability vector, but also on the ranking of the nodes according to their communicability values. However, in practice this does not occur. Through the analysis of the intersection similarity of several networks \cite{Acebron_SISC} it was shown that the chosen value of $\beta$ does not affect significantly the results, being in all cases the differences well below the typical error tolerances, and even
becoming smaller for increasingly larger network sizes. Consequently, and to ensure fast convergence of the method, in the simulations below we have used $\beta=1/\lambda_{max}$, where $\lambda_{max}$ is the maximum eigenvalue of A. However, finding the maximum eigenvalue for large networks is itself computationally costly and, in the following, a faster alternative 
based on computing the maximum degree of the network, $d_{max}$, was used instead as an upper bound value.


\subsection{Shared memory architecture}

The simulations corresponding to the shared memory architecture were run on both a commodity server equipped with $12$ cores and $32$ GB of RAM, and the MareNostrum supercomputer using a single node 
with $48$ cores.
The MLMC algorithm has been implemented in OpenMP, and to compare the performance with other methods, as well as to control the numerical errors, the MATLAB toolbox {\it funm$-$kryl} freely available  in \cite{funkryl} has been used. This method consists in the implementation of a Krylov subspace method with deflated restarting for matrix functions \cite{Higham2}. 
Note that Matlab was originally written in C/C++ and, specifically, operations involving matrix-vector multiplication or
matrix-matrix multiplication show nowadays an optimal performance in the latest versions of Matlab, since they are exploiting very efficiently 
multithreading execution as well as SIMD units available in current microprocessors. Taking into account that the Krylov subspace method requires 
matrix-vector multiplications extensively, we assume the 
obtained performance 
of the Matlab code to be more than competitive with respect to the performance of a native code in C/C++. Moreover, our implemented OpenMP code was not optimized 
to ensure a fair comparison with Matlab. Finally, it is worth remarking that the choice for using Matlab for comparison and not a native code was essentially motivated by the lack of any parallel code freely available in C/C++.

%


\noindent {\bf Example A: Partial Differential equations}.

The computational  time spent  by both, the  MC  and MLMC 
method, for solving the initial-boundary value problem  for a 3D heat equation at a single point is shown in Tables \ref{Table4a} 
and \ref{Table4b}. This has been done for different matrix sizes and number of cores running on the 
commodity server, and for about the same accuracy. It is worth observing that in view of the probabilistic nature of any Monte Carlo-based algorithm, the measured computational time for a single simulation  cannot be uniquely defined. Therefore,  in the following and for convenience, for both the Monte Carlo and MLMC method, the computational times reported in all tables have been chosen to be the most favorable simulation in terms of elapsed time, obtained after repeating the simulations a few times. Concerning  the  error, this  was  estimated  using the  aforementioned Krylov-based   method  by   setting   a  very   small   value  of   the stopping-accuracy  parameter,  $10^{-16}$,  as   well  as  the  restart parameter  to $40$.

For  comparison,  the  computational  time  spent  by  Matlab  is  also 
shown  only for  the  smaller matrix  size, since  for  the larger  one 
Matlab  simulations  run out  of  memory.  As  it  was pointed  out  in 
\cite{Acebron_SISC} this  is mainly  due to the  memory demands  of any 
Krylov-based algorithm.  Instead, the  Monte Carlo method  is extremely 
efficient  in terms  of memory  management, since  it requires  only to 
allocate in memory the input matrix.                                    

\begin{table*}[htbp]   
\small
	\caption{Elapsed time spent for computing the solution of the 3D heat equation at the single point $(0,0,0)$, and for time $t=1$ as a function of the number of cores. 
	The initial value function was $f({\bf x})= e^{-(x^2+y^2+z^2)}$.  
The accuracy  was kept fixed  to $5\times  10^{-4}$. The length  of the 
domain  was $\delta=4$  and  the  number of  grid  points was $n_x^3$,  with 
$n_x=256$.}                                                             

	\begin{center}
		{\tt
			\begin{tabular}{cccc}\hline
				{\bf Cores}
				&{\bf Time MC (s)}&{\bf Time MLMC (s)}&{\bf Time Matlab (s)}\\\hline
				$1$ & $231$ & $162$ & $428$ \\
				$4$ &  $64$ &  $45$ & $306$ \\
				$8$ &  $33$ &  $23$ & $327$ \\
				$12$&  $22$ &  $16$ & $334$ \\\hline
			\end{tabular}
		}
	\end{center}
	\label{Table4a}
\end{table*}

\begin{table*}[htbp]   
\small
	\caption{Elapsed time spent for computing the solution of the 3D heat equation at the single point $(0,0,0)$, and for time $t=1$ as a function of the number of cores. 
The accuracy  was kept fixed  to $5\times  10^{-4}$. The length  of the 
domain  was $\delta=4$  and  the  number of  grid  points was $n_x^3$,  with 
$n_x=512$.}                                                             
	\begin{center}
		{\tt
			\begin{tabular}{ccc}\hline
				{\bf Cores}
				&{\bf Time MC (s)}&{\bf Time MLMC (s)}\\\hline
				$1$ & $1245$ & $893$\\
				$4$ & $ 364$ & $247$\\
				$8$ & $ 184$ & $128$\\
				$12$& $ 122$ & $ 87$\\\hline
			\end{tabular}
		}
	\end{center}
	\label{Table4b}
\end{table*}

For the solution of the convection-diffusion equation, the arbitrary complex geometry plotted in Fig. \ref{geom} was used as the domain, being the Dirichlet boundary 
data chosen to be $u=0$ at the surface of the outer sphere, and $u=1$ at the surface of the inner cylinder. The size of the domain can be conveniently increased by 
simply rescaling both, the sphere and cylinder, using a single scale parameter {\it scale}.  To generate the computational mesh, and obtaining the corresponding FEM 
matrices and vector, the scientific software {\it COMSOL} \cite{COMSOL} was used, choosing specifically linear elements at the discretization setting. Concerning the 
element size used when meshing the geometry, it was kept fixed to be $0.8$ and $0.14$ for the maximum and minimum size, respectively.

\begin{figure}[!t]
\centering
\includegraphics[width=4.5in,angle=-90]{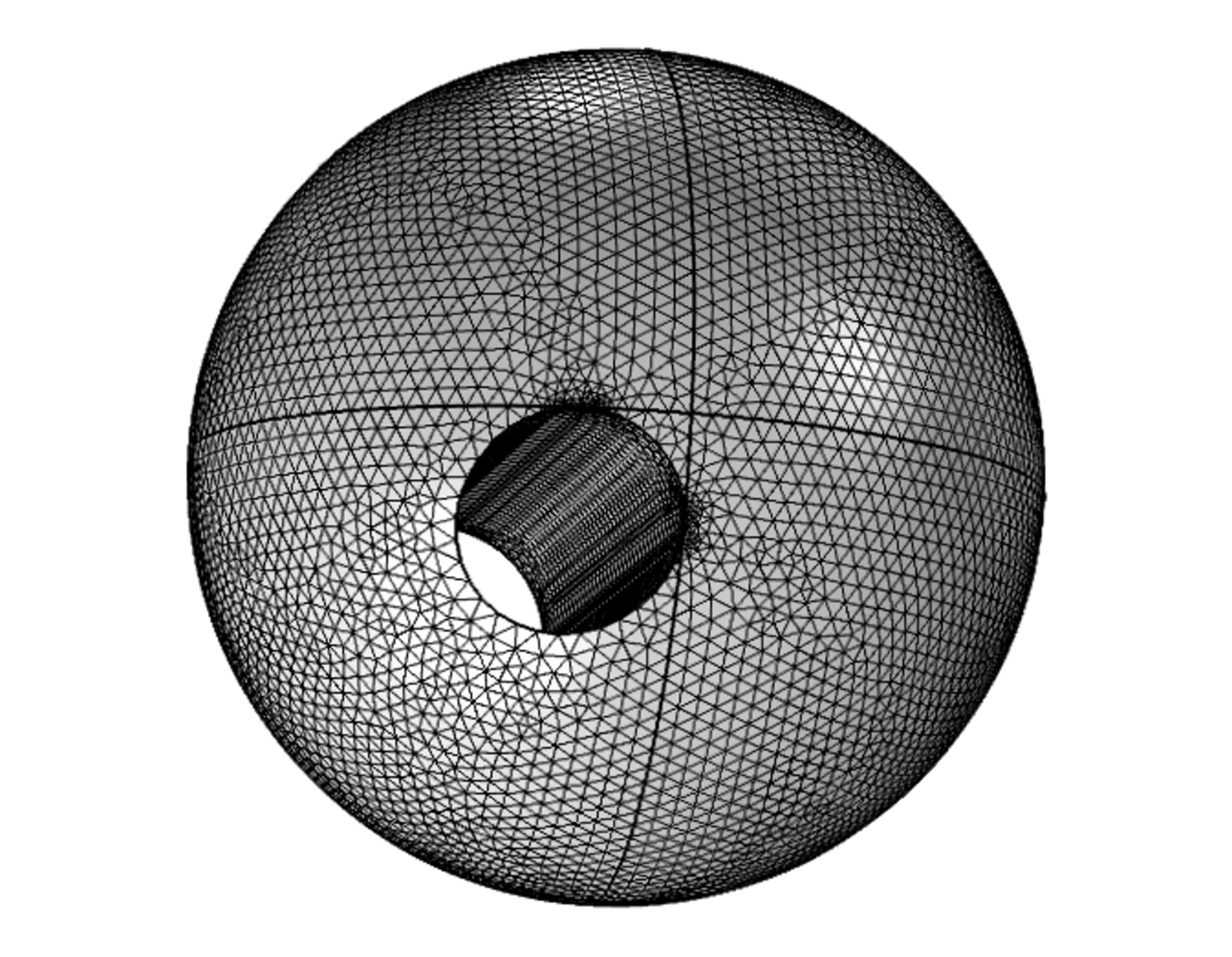}
\caption{Computational mesh describing the domain used for solving the 3D convection-diffusion equation.}
\label{geom}
\end{figure}

The computational time spent for computing the solution at a single point and time inside the domain for different number cores is shown in Table \ref{TableFEM1}. Note that 
both the MC method
and the MLMC method scale well with the number of cores, while the computational time spent with Matlab rapidly saturates when increasing the number of cores, due to the 
heavier intercommunication overhead of the Krylov-based algorithm.

\begin{table*}[htbp]   
\small
	\caption{Elapsed time spent for computing the solution of the 3D convection-diffusion equation at a single point, and for time $t=1$ as a function of the number 
	of cores. 
	The spatial point where the solution is computed,  has been chosen to be the nodal point of the computational mesh closer to the physical point $(0,0,0)$. The 
	accuracy  was kept fixed  
	to $5\times  10^{-4}$. The radius of the sphere was $r=4\times scale$, with $scale=12$, being the total number of nodes of the computational mesh $n=2,375,211$.}                                                             

	\begin{center}
		{\tt
			\begin{tabular}{cccc}\hline
				{\bf Cores}
				&{\bf Time MC (s)}&{\bf Time MLMC (s)}&{\bf Time Matlab (s)}\\\hline
				$1$ & $1,539$ & $520$ & $152$ \\
				$4$ &  $384$ &  $129$ & $107$ \\
				$8$ &  $214$ &  $83$ & $91$ \\
				$12$&  $170$ &  $61$ & $89$ \\\hline
			\end{tabular}
		}
	\end{center}
	\label{TableFEM1}
\end{table*}

In Table \ref{TableFEM2} the computational time spent when computing
the solution for different size domains is shown, being now the number of cores kept fixed to the maximum number of cores available. 

It is remarkable that the computational cost of the MC and MLMC method appears to be almost independent of the size of the domain, while it increases almost 
linearly for the Krylov-based method.
As it was already explained in \cite{Acebron_SISC} for the specific case of complex networks, this is mainly due to the similar matrix structure observed for any value of the matrix size. Because of this, the error becomes mostly independent of the size, and consequently it is not required to modify further the value of the sample size M, or the time step $\Delta t$ for increasingly larger matrix sizes (assuming a given prescribed accuracy for the solution), making therefore the computational cost of the algorithm almost independent of the size of the domain. This does not happen with the Krylov-based method, allowing specifically for the MLMC method to achieve a computational performance higher than the Matlab solution for large scale problems.

\begin{table*}[htbp]   
\small
	\caption{Elapsed time spent for computing the solution of the 3D convection-diffusion equation at a single spatial point, and time $t=1$ for different sizes of the 
	domain. This has been done rescaling both, the sphere and cylinder, choosing different values of the scale parameter.	The number of cores was kept fixed to 
	$12$ cores. The spatial point where the solution was computed consisted in the nodal points closer to the physical point $(0,0,0)$. The accuracy  was kept fixed  to $5\times  10^{-4}$. }                                                             

	\begin{center}
		{\tt
			\begin{tabular}{ccccc}\hline
				{\bf Scale}&{\bf n}
				&{\bf Time MC (s)}&{\bf Time MLMC (s)}&{\bf Time Matlab (s)}\\\hline
				$4$ & $83,813$ & $129$ & $55$ & $2$ \\
				$8$ & $69,4751$&  $167$ &  $59$ & $24$ \\
				$12$ & $2,375,211$  &  $170$ &  $61$ & $90$ \\\hline
			\end{tabular}
		}
	\end{center}
	\label{TableFEM2}
\end{table*}

Even though, the MC and MLMC methods were proposed initially to compute the solution at single temporal points, it turns out that they can 
be used as well to obtain the solution at intermediate instants of time. As a remarkable feature this can be done without any additional computational cost, as it was 
already explained in Sec. \ref{simulations}. For the Krylov-based methods, it is worth pointing out that there were also some recent attempts \cite{Mohy} to improve the performance of the method for computing the solution in a finite time interval, being however the performance of the resulting algorithm slightly worse than the performance of the algorithm for computing the solution at a single time. To test the accuracy of the obtained solution for intermediate times,  in Fig. \ref{fig_FEM_time} the solution computed using the Monte Carlo method is compared with the solution obtained using the Krylov-based method. Note the excellent agreement between both solutions for any value of time. 

\begin{figure}[!t]
\includegraphics[width=4.5in,angle=-90]{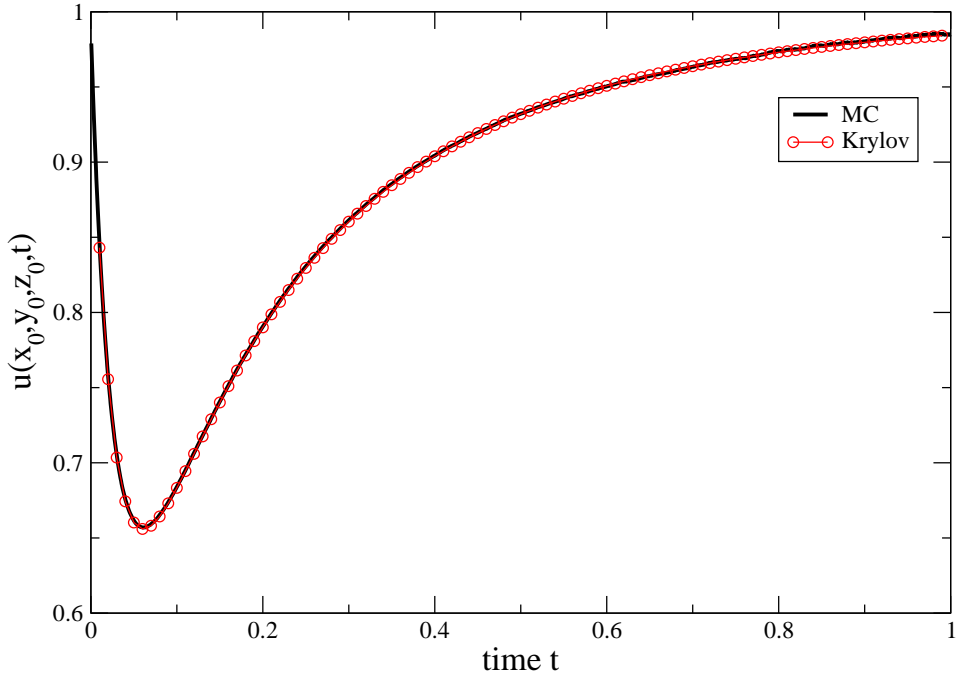}
\caption{Solution of the 3D convection-diffusion equation evaluated at the nodal points closer to the physical point $(0,0,0)$, and for different values of time. The initial value function was $f({\bf x})= e^{-(x^2+y^2+z^2)}$, and the velocity field $\beta$ $(-1,-1,-1)$. The solid line denotes the solution obtained with the MLMC method, and the dotted line corresponds to the Krylov-based solution.}
\label{fig_FEM_time}
\end{figure}


\noindent {\bf Example B: Complex networks}

{\it Small-world networks}. In Table \ref{Table1a} the computational time required to compute the total communicability of a small-world network of size $n=10^8$  is shown as a function 
of the number of cores for the Monte Carlo, MLMC method, and Matlab.

\begin{table*}[htb]   
\small
	\caption{Elapsed time spent for computing the total communicability of a small-world network as a function of the number of cores. The size of the matrix is $n=10^8$,  and the 
	accuracy $\varepsilon$ was kept fixed to $6.25 \times 10^{-4}$. The simulations were run on the commodity server.}
	\begin{center}
		{\tt
			\begin{tabular}{cccc}\hline
				{\bf Cores} &{\bf Time MC (s)}&{\bf Time MLMC (s)}&{\bf Time Matlab (s)}\\\hline
				$1 $ & $1027$ & $579$ & $348 $ \\
				$4 $ & $262 $ & $156$ & $258 $ \\
				$8 $ & $130 $ & $78$ & $257 $ \\
				$12$ & $86  $ & $51$ & $282 $ \\\hline
			\end{tabular}
		}
	\end{center}
	\label{Table1a}
\end{table*}

In Table \ref{Table1b} the results corresponding to a sort of weak scalability analysis of the MLMC algorithm are shown. For this purpose
the algorithm was run for an increasing number of cores, searching for the value of the accuracy $\varepsilon$ that equals the simulation time. Note that when the number of used cores increases, the accuracy $\varepsilon$ should be reduced accordingly. 
Moreover, since the computational cost of the MLMC algorithm is of order $\mathcal{O}(\varepsilon^{-2})$, the workload of the algorithm increases when reducing the value of $\varepsilon$, being therefore required to increase conveniently the number of used cores to keep approximately constant the overall execution time.  

From the results in Table  \ref{Table1b} it can be seen, for instance, that whenever the number of used cores increases $24$ times (passing from $1$ to $24$ cores), the value of $\varepsilon$ should be reduced by a factor of approximately $0.22$ for the same execution time. It is worth observing that for such a reduction of $\varepsilon$, the workload of the algorithm increases by a factor of $21$, which can be fully mitigated by increasing the number of used cores up to $24$. This is due to the remarkable scalability of the parallel algorithm. Similar conclusions can be drawn from the results obtained when using $48$ cores.

\begin{table*}[hbp]   
\small
	\caption{Weak scalability analysis of the MLMC algorithm for computing the total communicability of a small-world network. 
	The elapsed execution time was kept fixed around $1100$ seconds, and the size of the matrix was $n=10^8$. The simulations were run on the MareNostrum supercomputer.}
	\begin{center}
		{\tt
			\begin{tabular}{ccc}\hline
				{\bf Cores} &{\bf $\varepsilon$}& {\bf Time MLMC (s)}\\\hline
				$1 $ & $4.05\times10^{-4} $ & $1107$\\
				$24 $ & $8.8\times10^{-5} $ & $1095$\\
				$48 $ & $ 6.25\times10^{-5} $ & $1115$\\\hline
			\end{tabular}
		}
	\end{center}
	\label{Table1b}
\end{table*}

\begin{table*}[htbp]   
\small
\caption{Elapsed time spent for  computing the total communicability of 
a small-world network  as a function of the network  size. The accuracy 
$\varepsilon$ was kept fixed to $6.25\times 10^{-4}$, and the number of 
cores on the commodity server to $12$.}                                 

	\begin{center}
		{\tt
			\begin{tabular}{cccc}\hline
				{\bf Size n}
				&{\bf Time MC (s)}&{\bf Time MLMC (s)}&{\bf Time Matlab (s)}\\\hline
				$10^5 $ & $40$ & $30$ & $0.2 $ \\
				$10^6 $ & $80$ & $33$ & $1.9 $ \\
				$10^7 $ & $85$ & $48$ & $21 $ \\
				$10^8 $ & $86$ & $51$ & $282 $ \\\hline
			\end{tabular}
		}
	\end{center}
	\label{Table2}
\end{table*}

Table \ref{Table2} shows the results corresponding to the computational time when computing
the total communicability for different network sizes. Here the number of cores was fixed to the maximum number of cores available. 
It is remarkable to note that the computational cost of the MLMC method appears to be almost independent of the size of the network, while it increases almost linearly for the Krylov-based method.

{\it Scale-free networks}. In Table \ref{Table5} the results corresponding to a scale-free network for an arbitrarily large size are shown for different number of cores. Similar to the results obtained for the small-world network, the MLMC method outperforms
the Krylov-based method for large size networks and cores.

\begin{table*}[htbp]   
\small
	\caption{Elapsed time spent for computing the total communicability of a scale-free network as a function of the number of cores. The size of the matrix was $n=10^8$,  and the 
	accuracy $\varepsilon$ was kept fixed to $2.5 \times 10^{-8}$. The simulations were run on the commodity server. }
	\begin{center}
		{\tt
			\begin{tabular}{cccc}\hline
				{\bf Cores} &{\bf Time MC (s)}&{\bf Time MLMC (s)}&{\bf Time Matlab (s)}\\\hline
			
			$1 $ & $136$ & $79$ & $98 $ \\
                $4 $ & $ 35$ & $18$ & $88 $ \\
                $8 $ & $ 18$ & $ 9$ & $89 $ \\
                $12 $& $ 12$ & $ 6$ & $95 $ \\\hline 
			\end{tabular}
		}
	\end{center}
	\label{Table5}
\end{table*}

%

\subsection{Distributed memory architecture}

The simulations for a distributed memory architecture were carried out
on the MareNostrum Supercomputer of the Barcelona Supercomputing
Center (BSC) and on the Marconi Supercomputer at CINECA. In both
cases, two processes were launched on each node (one per processor),
with as many threads as physical cores available (24 threads on
MareNostrum and 18 on Marconi). Up to 200 nodes (a total of 9600
cores) were used on MareNostrum and up to 160 nodes (5760 cores) on
Marconi, which are respectively the maximum we had access to.

To the best of our knowledge, no parallel code suitable for distributed 
memory  architecture  capable  of  computing the  action  of  a  matrix 
exponential over  a vector  is currently  available. Therefore,  in the 
following, only results corresponding to the proposed multilevel method 
implemented in MPI are given.

\noindent {\bf Example B: Partial Differential equations}.

The computational time spent by the multilevel
method for computing the solution of the boundary value problem for
the 3D heat equations at a single point is shown in Table \ref{Table7}
for different number of cores. These results were all obtained in the
Marconi system. The speedup column indicates how much
faster the execution is relative to half the number of cores (previous
row in the table).

\begin{table*}[htbp]   
\small
	\caption{Elapsed time spent for computing the solution of the 3D
          heat equation at the single point $(0,0,0)$, and for 
          time $t=1$ as a function of the number of cores.  The accuracy
          $\varepsilon$ was kept fixed at $10^{-5}$. The length of the domain was $\delta=4$,
          and three different numbers of discretization points, $n_x$,
          were used. Note that the matrix size for the system is given
          by $n_x^3\times n_x^3$.}
	\begin{center}
		{
			\begin{tabular}{cccc}\hline
		{\bf $n_x$} & \bf Cores & \bf Time MLMC (s) & \bf Speedup\\\hline
				&  $720$ & $153$ &  \\
128				& $1440$ & $82$ & $1.9$ \\
				& $2880$ & $41$ & $2.0$\\
				& $5760$ & $21$ & $2.0$\\\hline
				&  $720$ & $774$ &  \\
256				& $1440$ & $395$ & $2.0$ \\
				& $2880$ & $196$ & $2.0$\\
				& $5760$ & $107$ & $1.8$\\\hline
				&  $720$ & $3577$ &  \\
512				& $1440$ & $1773$ & $2.0$ \\
				& $2880$ & $906$ & $2.0$\\
				& $5760$ & $467$ & $1.9$\\\hline
			\end{tabular}
		}
	\end{center}
	\label{Table7}
\end{table*}

In all cases, the speedup is very close to the ideal, even for such a
large number of cores. This is because most of the
calculations are totally independent, corresponding to the Monte Carlo
simulations performed at the each level of the method. For the defined
level of accuracy $\varepsilon$, a very large number of samples is required, exceeding the number
of $10^9$ for the coarsest level. Communication is required
between levels, but the overhead is negligible.

\noindent {\bf Example B: Complex networks} 

{\it Small-world networks}. In Table \ref{Table6} the computational
time required to compute the total communicability of a small-world
network of size $n=10^8$ is shown as a function of different number of
cores for the multilevel method.

\begin{table*}[htbp]   
\small
	\caption{Elapsed time spent for computing the total communicability of a small-world network as a function of the number cores. The size of the matrix was $n=10^8$,  and the 
	accuracy $\varepsilon$ was kept fixed at $10^{-7}$. }
	\begin{center}
		{
			\begin{tabular}{cccc}\hline
			  & \bf Cores & \bf Time MLMC (s) & \bf Speedup\\\hline
				& $1200$ & $315$ & \\
MareNostrum			& $2400$ & $175$ & $1.8$ \\
				& $4800$ &  $87$ & $2.0$ \\ 
				& $9600$ &  $50$ & $1.7$ \\ 
\hline
				&  $720$ & $320$ & \\
Marconi		                & $1440$ & $166$ & $1.9$\\
				& $2880$ &  $86$ & $1.9$\\ 
				& $5760$ &  $44$ & $2.0$\\ 
\hline
			\end{tabular}
		}
	\end{center}
	\label{Table6}
\end{table*}

As in the case of the partial differential equation, the scalability of the method is
almost perfect.

\section{Conclusion}\label{conclusions}

The multilevel Monte Carlo method was conveniently recast to be able to compute the action of a matrix exponential over a 
vector. As the main ingredient of the method, the leading probabilistic method requires generating suitable random paths which evolve through the indices of the matrix 
according to the probability law of a continuous-time Markov chain governed by the associated Laplacian matrix. 

This new method extends the previous work in three respects. First, the probabilistic method proposed in \cite{Acebron_SISC} has been generalized allowing now to be 
applied to any class of matrices (not only adjacency matrices). Second, it allows now to compute much more efficiently a highly accurate solution. In fact the computational 
complexity has been
proved in this paper to be significantly better than that of the classical Monte Carlo method. Third, the underlying algorithm after parallelization has been shown to be highly scalable, which in practice enables simulation of large-scale problems
for extremely large number of cores. We analyzed the performance of the algorithm running several benchmarks of interest in science and engineering. These consist in computing the total communicability of the network for a variety of complex networks (real and synthetic), and in 
solving at single points inside the domain a boundary-value problem for parabolic partial differential equations. Finally, whenever available, 
simulations based on a standard Krylov-based method have been conducted, and the performance compared with the multilevel MC method. In particular, the multilevel MC method clearly outperforms the deterministic method for solving 
problems consisting in large matrices, not only in terms of computational time, but also in terms of memory requirements. 

To conclude, an interesting question deserving further investigation is whether the proposed method can be extended to deal with other matrix functions such as trigonometric functions arising in oscillatory problems, and even hyperbolic functions appearing in coupled hyperbolic systems of partial differential equations.

\section*{Acknowledgments}

\begin{sloppy}

The   work   has  been   performed   under   the  Project   HPC-EUROPA3 
(INFRAIA-2016-1-730897), with the support of the EC Research Innovation 
Action under the H2020 Programme; in particular, the authors gratefully 
acknowledge  the support  of  the Computer  Architecture Department  at 
Universitat Polit\`ecnica de Catalunya (UPC) and the computer resources 
and  technical  support  provided by  Barcelona  Supercomputing  Center 
(BSC).  We acknowledge  PRACE  for  awarding us  access  to Marconi  at 
CINECA,  through grant  2010PA4246.  This work  was  also supported  by 
Funda\c{c}\~{a}o  para a  Ci\^{e}ncia e  a Tecnologia  under Grant  No. 
UID/CEC/50021/2019, by  the Spanish Ministry of  Science and Technology 
through  TIN2015-65316-P project  and by  the Generalitat  de Catalunya 
(contract 2017-SGR-1414).                                               

\end{sloppy}

\bibliographystyle{abbrv}

\begin{thebibliography}{10}


\bibitem{Jahnke} T. Jahnke, and C. Lubich, \newblock {Error bounds for exponential operator splittings},\newblock{\em BIT},40 (2000) 735-744.

\bibitem{Merris} R. Merris, \newblock {Laplacian matrices of graphs: A survey}, \newblock{\em Linear Algebra and Its Applications}, 197 (1994) 143-176.


\bibitem{Acebron1} J.A.~Acebr\'on, M.P.~Busico, P.~Lanucara, and R.~Spigler,\newblock 
{Domain decomposition solution of elliptic boundary-value problems}
\newblock {\em SIAM J. Sci. Comput},  27 (2005) 440-457.

\bibitem{Acebron2} J.A.~Acebr\'on, and A. Rodr\'{i}guez-Rozas, 
\newblock {A new parallel solver suited for arbitrary semilinear parabolic
  partial differential equations based on generalized random trees}, \newblock {\em  Journal of Computational Physics}, 230 (2011) 
  7891-7909.

\bibitem{Acebron3} J.A.~Acebr\'on, and A. Rodr\'{i}guez-Rozas, \newblock {Highly efficient numerical algorithm based on 
random trees for accelerating parallel Vlasov-Poisson simulations},
\newblock {\em  Journal of Computational Physics}, 250 (2013) 224-245.

\bibitem{Acebron4} S. Mancini, F. Bernal, and J.A.~Acebr\'on, \newblock{An Efficient Algorithm for Accelerating Monte Carlo Approximations of the Solution to Boundary Value Problems},
\newblock{\em J. Sci. Comput.}, 66 (2016) 577-597.

\bibitem{Acebron_SISC} J. A. Acebr\'on, \newblock{A Monte Carlo method for computing the action of a matrix exponential on a vector}, \url{https://arxiv.org/abs/1904.12759}, 
Appl. Math. Comput. (2019) in press.

\bibitem{Anderson} D.F Anderson, and D.J. Higham, \newblock{Multilevel Monte Carlo for continuous time Markov chains, with
applications in biochemical kinetics}, \newblock {\em Multiscale Model. Simul.}, 10 (2012) 146-179.

\bibitem{Waymire} R.N. Bhattacharya, and  E.C. Waymire, \newblock{Stochastic Processes with Applications}, \newblock{SIAM}, 2009.

\bibitem{Benzi} M. Benzi, E. Estrada, and C. Klymko,\newblock{Ranking hubs and authorities using matrix functions},\newblock{\em Linear Algebra and Its Applications},
438 (2013) 2447-2474.

\bibitem{Benzi2} M. Benzi, and C. Klymko, \newblock{Total communicability as a centrality measure}, \newblock{\em J. Complex Networks}, 1 (2013) 124-149.

\bibitem{Benzi3} M. Benzi, T.M. Evans, S.P. Hamilton, M.L. Pasini, and S.R. Slattery,\newblock{Analysis of Monte Carlo accelerated iterative methods for sparse linear systems},
\newblock{\em Numerical Linear Algebra with Appl.}, 24 (2017).


\bibitem{Botchev} M.A. Botchev, V. Grimm, and M. Hochbruck, \newblock{Residual, restarting, and Richardson iteration for the matrix
exponential}, \newblock{\em SIAM J. Sci. Comput}, 35 (2013) A1376-A1397

\bibitem{COMSOL} \url{http://www.comsol.com/}

\bibitem{CONTEST} \url{http://www.maths.strath.ac.uk/research/groups/numerical_analysis/contest}

\bibitem{dimov} I.T. ~Dimov, \newblock{Monte Carlo Methods for Applied Scientists},\newblock{World Scientific}, 2008.

\bibitem{dimov2} I. T. ~Dimov, T. T. Dimov, and T. V. Gurov,\newblock{A new iterative Monte Carlo Approach for Inverse Matrix Problem},\newblock{\em J. Comput. Appl. Math.}, 92 (1998) 15-35.

\bibitem{dimov3} I. T. Dimov, V.N. Alexandrov, and A. Karaivanova, {\newblock Parallel resolvent Monte Carlo algorithms for linear algebra problems}, 
\newblock{\em Mathematics and Computers in Simulation}, 55 (2001) 25-35.

\bibitem{dimov4} I. Dimov, S. Maire, and J.M. Sellier,{\newblock A new Walk on Equations Monte Carlo method for solving systems of linear algebraic equations},\newblock{\em Applied Mathematical Modelling}, 39 (2015) 4494-4510.

\bibitem{Evans} M.~Evans, and T.~Swartz, \newblock 
{Approximating Integrals Via Monte Carlo and Deterministic Methods}, \newblock{Oxford University Press}, 2000.

\bibitem{Estrada_review} E. Estrada, N. Hatano, and M. Benzi, \newblock{The physics of communicability in complex networks}, \newblock{\em Physics Reports} 514 (2012) 89-119. 

\bibitem{Forsythe} G. Forsythe, and R. Leibler, \newblock{\em Matrix inversion by a Monte Carlo method}, Math. Tables Other Aids Comput., 4 (1950) pp. 127-129.

\bibitem{funkryl} \url{http://www.mathe.tu-freiberg.de/∼guettels/funm kryl/}

 \bibitem{Giles1} M.B. Giles, \newblock{Multilevel Monte Carlo methods}, \newblock{\em Acta Numerica}, 24 (2015)  259-328.
 
 \bibitem{Giles2} M.B. Giles, \newblock{Multilevel Monte Carlo path simulation}, \newblock{\em Operations Research}, 56 (2008) 607-617.
 
 \bibitem{Higham} N.J.~Higham, and A. H.~Al-Mohy, \newblock{Computing matrix functions}, \newblock{\em Acta Numerica}, 19 (2010) 159-208.
 
 \bibitem{Higham2} N.J.~Higham, and A. H.~Al-Mohy, \newblock{Functions of matrices: Theory and Computation}, \newblock{SIAM }, 2008
 

 \bibitem{Martinez} A. Martinez, L. Bergamaschi, M. Caliari, and  M. Vianello, {\newblock A massively parallel exponential integrator for advection-diffusion models},
 {\newblock J. Comput. Appl. Math.} 231 (2009) 82–91.

\bibitem{Mascagni} H. Ji, M. Mascagni, and Y. Li, \newblock{Convergence Analysis of Markov Chain Monte Carlo Linear Solvers Using Ulam--von Neumann Algorithm},
\newblock{\em SIAM J. Numer. Anal.}, 51 (2013) 2107-2122.

\bibitem{Mattheij} R.M.M.~Mattheij, S.W.~Rienstra, and J.H.M.~ten Thije Boonkkamp,\newblock {Partial Differential Equations: Modeling,
Analysis, Computation}, \newblock {\em SIAM monographs},2005.

\bibitem{Mazunder} S. Mazunder, \newblock{Numerical Methods for Partial Differential Equations}, \newblock{Academic Press}, 2015.


\bibitem{Mohy}  A. H.~Al-Mohy, and  N.J.~Higham, \newblock{Computing the action of the matrix exponential, with
an application to exponential integrators}, \newblock{SIAM J. Sci. Comput.} 33 (2011) 488-511.  

\bibitem{Okten} G. \"Okten,\newblock{\em Solving linear equations by Monte Carlo simulation}, \newblock{SIAM J. Sci. Comput.} 27 (2005) 511-531.


\bibitem{Pusa} M. Pusa, and J. Lepp\"anen,\newblock{Computing the Matrix Exponential in Burnup Calculations},\newblock{\em Nuclear Sci. and Eng},164 (2010) 140-150.


\bibitem{Sadiku} M.N.O.~Sadiku, \newblock {Monte Carlo methods for electromagnetics},
\newblock {\em CRC press}, 2009.

\bibitem{Sidjea} R.B. Sidjea, and W.J. Stewart, \newblock{A numerical study of large sparse matrix exponentials arising in Markov chains}, \newblock{\em Comput. Stat. Data Anal.}, 
29 (1999) 345-368.

\bibitem{tamu} \url{https://sparse.tamu.edu/}

\bibitem{top500} \url{https://www.top500.org/}

\bibitem{Cheng1} S.H. Weng, Q. Chen, and C.K. Cheng,\newblock{Circuit Simulation by Matrix Exponential Method},\newblock{\em IEEE ASIC Conference}, (2011) 369-372. 

\bibitem{Cheng2} H. Zhuang, S.H. Weng, and C.K Cheng, \newblock{Power Grid Simulation using Matrix Exponential Method with Rational Krylov Subspaces}, 
\newblock{\em IEEE ASIC Conference}, (2013).

\bibitem{Zienkiewicz} O.C Zienkiewicz, R.L. Taylor, and J.Z Zhu, \newblock{The Finite Element Method: Its basis and fundamentals}, \newblock{Elsevier }, 2005.
\end{thebibliography}

\end{document}